\documentclass{amsart}
\usepackage{amsmath}
\usepackage{amssymb}
\usepackage{amsfonts}
\usepackage{wasysym}
\usepackage{pxfonts}

\input txdtools

\newtheorem{thm}{Theorem}[section]

\newtheorem{lem}[thm]{Lemma}
\newtheorem{prop}[thm]{Proposition}
\theoremstyle{definition}
\newtheorem{defn}[thm]{Definition}

\theoremstyle{remark}
\newtheorem{rem}[thm]{Remark}
\newtheorem{fact}[thm]{Fact}

\numberwithin{equation}{section}
\numberwithin{figure}{section}

\renewcommand{\d}{{\partial}}
\newcommand{\dbar}{\overline{\partial}}
\newcommand{\e}{\mathrm e}

\newcommand{\C}{{\mathbb C}}
\newcommand{\D}{{\mathbb D}}
\newcommand{\T}{{\mathbb T}}
\newcommand{\R}{{\mathbb R}}

\newcommand{\1}{{\mathbf 1}}

\newcommand{\wt}{\widetilde}
\newcommand{\wh}{\widehat}

\newcommand{\dA}{{\diff A}}

\newcommand{\re}{\operatorname{Re}}
\newcommand{\im}{\operatorname{Im}}

\newcommand{\diff}{{\mathrm d}}
\newcommand{\imag}{{\mathrm i}}

\newcommand{\dist}{\operatorname{dist}}

\newcommand{\eps}{\varepsilon}

\newcommand{\trace}{\operatorname{trace}}

\newcommand{\bfK}{\mathbf{K}}

\newcommand{\erfc}{\operatorname{erfc}}

\def\lpar{\left (}
\def\rpar{\right )}
\def\labs{\left |}
\def\rabs{\right |}
\def\babs#1{\labs {#1} \rabs}

\begin{document}

\title[Beurling--Landau densities of weighted Fekete sets]{Beurling--Landau densities of weighted Fekete sets and correlation kernel estimates}
\subjclass[2010]{31C20; 82B20; 30E05; 94A20}
\keywords{Weighted Fekete set; droplet; equidistribution; concentration operator; correlation kernel}

\author{Yacin Ameur}

\address{Yacin Ameur\\ Department of Mathematics\\
Lule{\aa} University of Technology\\
971 87 Lule{\aa}\\
Sweden}

\email{yacin.ameur@gmail.com}

\author{Joaquim Ortega-Cerd\`{a}}

\address{Joaquim Ortega-Cerd\`{a}\\ Departament de Matemàtica Aplicada i An\`{a}lisi\\ Universitat de Barcelona\\ Gran Via 585\\ 08007 Barcelona\\ Spain}

\email{jortega@ub.edu}

\thanks{This work is a contribution to the research program on "Complex Analysis and Spectral Problems'' which was conducted at the CRM in Barcelona during
the spring semester of 2011. The first author was supported by grants from Magnussons fond, Vetenskapsr{\aa}det, SveFum, and the European Science Foundation. The second author was supported by grants
MTM2008-05561-C02-01 and 2009 SGR 1303.}

\begin{abstract} Let $Q$ be a suitable real valued function on $\C$ which increases sufficiently rapidly as $z\to\infty$.
An $n$-Fekete set corresponding to $Q$ is a subset $\{z_{n1},\ldots,z_{nn}\}$ of $\C$ which maximizes the weighted Vandermonde determinant
$\prod_{i<j}^n\babs{z_{ni}-z_{nj}}^2\e^{-n(Q(z_{n1})+\cdots+Q(z_{nn}))}$. It is
well known that there exists a
compact set $S$ known as the "droplet'' such that the sequence of measures
$\mu_n=n^{-1}(\delta_{z_{n1}}+\cdots+\delta_{z_{nn}})$
converges to the equilibrium measure $\Delta Q\cdot \1_S~\dA$ as $n\to \infty$. In this note we consider a related topic, proving that Fekete sets are in a sense maximally spread out with respect to the equilibrium measure. In general, our results apply only to a part of the Fekete set, which is at a certain distance away from the boundary of the droplet. However, for the Ginibre potential $Q=\babs{z}^2$ we obtain results which hold globally; we conjecture that such global results are true for a wide range of potentials.
\end{abstract}

\maketitle

In this paper we discuss equidistribution results for weighted Fekete sets in subsets of the plane.
More precisely, we show that
 Fekete sets are maximally spread out relative to a rescaled version of the Beurling--Landau density, in the "droplet" corresponding to the given weight.  Our method combines Landau's idea to relate the density of a family of discrete sets
 to properties of the spectrum of the concentration operator, with estimates for the correlation kernel of the corresponding random normal matrix ensemble.

\section{Fekete sets}

\subsection{Potentials and droplets} We recapture some notions and results from weighted potential theory. Proofs
and further results can be found in \cite{ST}. Cf. also \cite{AHM} and \cite{HM2} where the setting is more tuned to fit the present discussion.

Let $Q:\C\to \R\cup\{+\infty\}$ be a suitable function (the "potential'' or
"external field'') satisfying
\begin{equation*}\liminf_{z\to\infty}\frac {Q(z)}{\log\babs{z}^2}=+\infty.\end{equation*}
(In detail: we require in addition that the function $w:=\e^{-Q/2}$ satisfy the mild condition of being an "admissible weight'' in the
sense of \cite{ST}, p. 26. This means that $w$ is upper semi-continuous and the set $\{w>0\}$ has positive logarithmic capacity.)

We associate to $Q$ the "equilibrium potential'' $\wh{Q}$ in the following way: Let $\text{SH}_Q$ be the set of all subharmonic functions $u:\C\to\R$ such that $u(z)\le \log_+ \babs{z}^2+\text{const.}$ and $u\le Q$ on $\C$.
One defines the equilibrium potential as
$\wh{Q}(z)=\sup\{u(z)~;~u\in\text{SH}_Q\}.$
The droplet associated to $Q$ is the set
$$S=\{z\in\C;\, Q(z)=\wh{Q}(z)\}.$$
This is a compact set; one has that $\Delta Q\ge 0$ on $S$ and that the \textit{equilibrium measure}
\begin{equation}\label{equi}\diff \sigma(z)=\1_S(z) \Delta Q(z) \diff A(z)\end{equation}
is a probability measure on $\C$. Here we agree
that $\diff A$ is normalized area measure
$\diff A=\frac  1 \pi {\diff x\diff y} ,$
while $\Delta =\d\dbar=\frac 1 4 (\d^2/\d x^2+\d^2/\d y^2)$ is the normalized Laplacian; $\d=\frac 1 2 (\d/\d x-\imag \d/\d y)$
and $\dbar=\frac 1 2 (\d/\d x+\imag \d/\d y)$ are the complex derivatives.

We will make the standing assumption that $Q$ be $C^3$-smooth and strictly
subharmonic in some neighbourhood $\Lambda$ of $S$.
In other words, we assume that
the \textit{conformal metric}
$\diff s^2(z)=\Delta Q(z)\babs{\diff z}^2$ is comparable to the Euclidean
metric on $\Lambda$.

\subsection{Fekete sets} Consider the weighted Vandermonde determinant
\begin{equation*}V_n(z_1,\ldots,z_n)=\prod_{i<j}\babs{z_i-z_j}^2\e^{
-n(Q(z_1)+\cdots+Q(z_n))},\quad z_1,\ldots, z_n\in \C.\end{equation*}
A set $\mathcal{F}_n=\{z_{n1},\ldots,z_{nn}\}$ which maximizes $V_n$ is called an \textit{$n$-Fekete set} corresponding to $Q$. Notice that Fekete sets are not unique.

Equivalently, the set $\mathcal{F}_n$ minimizes the weighted energy
\begin{equation}\label{ham}H_n(z_1,\ldots,z_n)=\sum_{i\ne j}\log
\babs{z_i-z_j}^{-1}+n\sum_{j=1}^n Q(z_j)\end{equation}
over all configurations $\{z_j\}_{j=1}^n\subset\C$.
If we think of the points $z_j$ as giving locations for $n$ identical repelling
point charges with total charge $1$ confined to $\C$ under the influence of the
external magnetic field $nQ$, then $H_n$ can be regarded as the the energy of the
system.

The following classical result displays some fundamental and well-known properties of Fekete sets.

\begin{thm} \label{eqcon} For any Fekete set $\mathcal{F}_n=\{z_{n1},\ldots,z_{nn}\}$ holds:
\begin{enumerate}
\item $\mathcal{F}_n\subset S$
\item Let $\sigma$ be the equilibrium measure \eqref{equi}.
We then have convergence in the sense of measures
$$\frac 1 n \sum_{j=1}^n \delta_{z_{nj}}\to \sigma,\quad \text{as}\quad n\to\infty.$$
\end{enumerate}
\end{thm}

A proof can be found in \cite{ST}, theorems III.1.2 and III.1.3. (Notice that our assumptions on $Q$ imply that
$S=S^*$ in the notation of \cite{ST}.)  The theorem \ref{eqcon} was generalized to line bundles over complex manifolds in \cite{BB}, \cite{BBW}.

We remark that the property (1) is essential to the analysis in this paper, and that the standard proof of (1) (e.g. in \cite{ST}) depends on the "maximum principle for weighted polynomials'', which is reproduced in Lemma \ref{BAHM} below.

We will consider related questions concerning the distribution of Fekete points. In a sense, we will prove that these points are maximally spread out with respect to the conformal metric. To quantify this assertion, we introduce some definitions.

\begin{defn} Let $\mathcal{F}=\{\mathcal{F}_n\}_{n=1}^\infty$ be a family of $n$-Fekete sets. Also let
$\zeta=(z_n)_1^\infty$ be a sequence of points in $S$.
We define the \textit{lower Beurling--Landau's density}
of $\mathcal{F}$ with respect to $\zeta$ by
$$D^-\lpar \mathcal{F};\zeta\rpar=\liminf_{R\to\infty}\liminf_{n\to\infty}\frac {\#\lpar \mathcal{F}_n\cap D\lpar z_n;R/\sqrt{n}\rpar\rpar}
{R^2\Delta Q(z_n)},$$
and we define the corresponding upper density by
$$D^+\lpar \mathcal{F};\zeta\rpar=\limsup_{R\to\infty}\limsup_{n\to\infty}\frac {\#\lpar \mathcal{F}_n\cap D\lpar z_n;R/\sqrt{n}\rpar\rpar}
{R^2\Delta Q(z_n)}.$$
We also put
\begin{equation*}d_n(\zeta)=\dist(z_n,\C\setminus S).\end{equation*}
Here "dist'' denotes the Euclidean distance in the plane, and $D(\zeta;r)$ is
the open disk with center $\zeta$
and radius $r$.
\end{defn}

We have the following theorem.

\begin{thm}\label{MTH1} Put
$\delta_n=\log^2 n/\sqrt{n},$ and
suppose that
$d_n(\zeta)\ge 3\delta_n$
for all $n$. Then
\begin{equation}\label{ludens}D^-\lpar \mathcal{F};\zeta\rpar=D^+\lpar \mathcal{F};\zeta\rpar=1.\end{equation}
\end{thm}

A proof is given in \textsection \ref{homolu}.

\begin{rem} The function $\varrho_n(z)^{-2}$
defined by
$n\sigma(D(z;\varrho_n(z)))=1$ can be considered as a regularized version of the Laplacian $\Delta Q(z)$.
Replacing $\Delta Q(z_n)$ by $\varrho_n(z_n)^{-2}$ in our definition of Beurling--Landau's densities, it becomes possible to
extend our results to cover some situations in which $\Delta Q=0$ at isolated points of the droplet.
\end{rem}

\subsection{The Ginibre case}
The potential $Q(z)=\babs{z}^2$ is known as the \textit{Ginibre potential}. It is easy to see that for this potential,
the droplet is $S=\overline{\D}$, i.e. the closed unit disk with center $0$.

\begin{thm}\label{MTH2} Suppose that $Q(z)=\babs{z}^2$.
Let $\zeta=(z_n)$ be a sequence in $\overline{\D}$ and assume that the limit
$L=\lim_{n\to \infty}\sqrt{n}(1-\babs{z_n})$
exists.
Then
\begin{enumerate}
\item If $L=+\infty$, then \eqref{ludens} holds
\item If $L<+\infty$, then
\begin{equation}\label{ludens2}D^-\lpar \mathcal{F};\zeta\rpar=D^+\lpar \mathcal{F};\zeta\rpar=\frac 1 2.\end{equation}
\end{enumerate}
\end{thm}

A proof is given in \textsection \ref{gincas}.

\begin{rem} The condition that the limit $L$ exists is really superfluous and is made
merely for technical convenience. Indeed, we can assert that $\liminf_{n\to\infty} \sqrt{n}(1-\babs{z_n})=+\infty$ then \eqref{ludens} holds while if $\limsup_{n\to\infty}\sqrt{n}(1-\babs{z_n})<+\infty$ then \eqref{ludens2} holds.
These somewhat more general statements can be proved
without difficulty by using the arguments below.
\end{rem}

\subsection{A conjecture} The boundary of a droplet corresponding to a smooth potential
is in general a quite complicated set. However, owing to Sakai's theory \cite{Sa}, it is known that the situation is more manageable for potentials
$Q$ which are \textit{real-analytic} in a neighbourhood of the droplet.
Namely,
for a real analytic potential $Q$, the boundary of $S$ is a finite union of real analytic arcs and
possibly a finite number of isolated points. The boundary of $S$ may also have
finitely many singularities which can be either cusps or
double-points. This result can easily be proved using arguments from \cite{HS}, Section 4.

Suppose that $Q$ is real-analytic and strictly subharmonic in a neighbourhood of $S$, and assume that $\d S$ has no singularities. Let $S^*$ denote the set $S$ with eventual singularities and isolated points
removed.
Also let $\zeta=(z_n)_1^\infty$ be a sequence of points in $S^*$ and assume for simplicity that the limit
$L=\lim_{n\to\infty}\sqrt{n}d_n(\zeta)$ exists, where $d_n(\zeta)$ is the
distance of $z_n$ to $\d S$. We conjecture that for any sequence
$\mathcal{F}=\{\mathcal{F}_n\}$ of weighted Fekete sets, we have
(i) if $L=+\infty$, then $D^-(\mathcal{F},\zeta)=D^+(\mathcal{F},\zeta)=1$ and (ii)
if $L<+\infty$, then $D^-(\mathcal{F},\zeta)=D^+(\mathcal{F},\zeta)=1/2$.

The conjecture is supported by the results of the forthcoming paper \cite{AKM}.

\subsection{Earlier work and related topics} The topics considered in this note, as well as our basic strategy,
were inspired by
the paper \cite{L} by Landau, which concerns questions about
interpolation and sampling for functions in Paley--Wiener spaces. In particular,
our "Beurling--Landau densities'' can be seen as straightforward adaptations of
the densities defined in \cite{L}, and our results below are parallel to those
of Landau. The historically
interested reader should also consult Beurling's lecture notes (see the references in \cite{L}), where some of the basic concepts appeared earlier; in fact Landau's exposition depends in an essential way on Beurling's earlier work.

In the one-component plasma (or "OCP'') setting, one introduces a temperature $1/\beta$, where $\beta>0$. The probability measure
$\diff \mathbf{P}_n^\beta(z)=(Z_n^\beta)^{-1}\e^{-\beta H_n(z)}\diff V_n(z)$ on $\C^n$ is known as the density of states
at the temperature $1/\beta$. Here $\diff V_n$ is Lebesgue measure on $\C^n$, $H_n$ is the Hamiltonian \eqref{ham}, and
$Z_n^\beta$ is a normalizing constant. One then considers configurations $\Psi_n^\beta=\{z_i\}_1^n$ picked randomly with
respect to $\mathbf{P}_n^\beta$.

Intuitively, Fekete sets should correspond to particle configurations at temperature zero, or rather, the "limiting
configurations'' as $1/\beta\to 0$, although the latter "limit'' so far has been understood mostly on a physical level.
In this interpretation, the methods of the present note prove that the Beurling--Landau density of temperature zero configurations is in fact completely determined by properties at $\beta=1$. (More precisely: it is determined by the one- and two-point functions of $\mathbf{P}_n^1$.)

A more subtle problem is to characterise Fekete sets amongst all configurations of Beurling--Landau density one. It is believed that a certain crystalline structure  will manifest itself (known as the "Abrikosov lattice''). We refer to \cite{Fal}, \cite{SS} and the references therein for further details on this topic. A survey of related questions for minimum energy points on manifolds is found in \cite{HSa}.

\section{Weighted polynomials and triangular lattices}

Our approach combines the method for characterizing Fekete sets and triangular lattices from the papers \cite{MO} and \cite{OP} with correlation kernel estimates of the type found in \cite{B}, \cite{A}, \cite{AHM}, \cite{AHM3}.
In the Ginibre case, we use the explicit representation of the correlation kernel available for that potential, as well as
estimates from the papers \cite{Sz}, \cite{FH}, \cite{BS}, and \cite{BM}.

\subsection{Weighted polynomials} Let $H_{n}$ be the space of polynomials $p$ of degree at most $n-1$, normed by
$\left\|p\right\|_{nQ}^2:=\int_\C \babs{p(z)}^2\e^{-nQ(z)}\dA(z).$
The reproducing kernel for $H_n$ is
$K_n(z,w)=\sum_{j=0}^{n-1} e_j(z)\overline{e_j(w)},$
where $\{e_j\}_0^{n-1}$ is an orthonormal basis for $H_n$.

For our purposes, it is advantageous to work with spaces $\tilde{H}_n$ of weighted polynomials
$f=p\cdot \e^{-nQ/2}$, where $p$ is a polynomial of degree
$\le n-1$, and one defines the norm in $\tilde{H}_n$ as the usual $L^2(\dA)$-norm.
The reproducing kernel for $\tilde{H}_n$ is given by
$$\mathbf{K}_n(z,w)=K_n(z,w)\e^{-nQ(z)/2-nQ(w)/2}.$$
The function $\mathbf{K}_n$ is known as the \textit{correlation kernel} corresponding to the potential $Q$; the reproducing
property means that
$$f(z)=\langle f,\mathbf{K}_{n,z}\rangle,\qquad f\in\tilde{H}_n,~z\in\C,$$
where $\mathbf{K}_{n,z}(\zeta)=\mathbf{K}_n(\zeta,z)$, and the inner product is the usual one in $L^2=L^2(\C,\dA)$.

When $\rho n$ is not an integer, we interpret $\tilde{H}_{\rho n}$ as the
space $H_k$ where $k$ is the largest integer satisfying $k<\rho n$. All statements below shall be understood in terms of this convention; in particular,
$\mathbf{K}_{\rho n}(z,w):=K_k(z,w)\e^{-k(Q(z)+Q(w))/2}.$

\subsection{Triangular lattices}
Let $\mathcal{Z}=\left\{\mathcal{Z}_j\right\}_{j=1}^\infty$ be a triangular lattice of points in $\C$. We write
$$\mathcal{Z}_n=\{z_{n1},z_{n2},\ldots,z_{nm_n}\}.$$
It will be convenient to introduce some classes of lattices.

Let $\rho>0$.
A family $\mathcal{Z}$ is said to be $\rho$-\textit{interpolating} if
there is some constant $C$ such that, for all families of values $c=\{c_n\}_1^\infty$,
$c_n=\{c_{nj}\}_{j=1}^{m_n}$,
such that
$$\sup_n \frac 1 {n\rho} \sum_{j=1}^{m_n} \babs{c_{nj}}^2<\infty,$$
there exists a sequence
$f_n\in \tilde{H}_{\rho n}$ such that
$f_n(z_{nj})=c_{nj}$, $1\le j\le m_n$, and
\begin{equation*}\left\|f_n\right\|^2\le C \frac 1 {n\rho} \sum_{j=1}^{m_n}\babs{c_{nj}}^2.\end{equation*}

We say that a family $\mathcal{Z}$ is \textit{uniformly separated} if there is a number $s>0$ such that
for any two distinct points $z,w\in \mathcal{Z}_n$ we have
$\babs{z-w}> s/\sqrt{n}$. The following simple lemma holds.

\begin{lem} \label{BASE0} Any interpolating family which is
contained in $S$ is uniformly separated.
\end{lem}

A proof is given in \textsection \ref{BASE1}.

Intuitively, an interpolating family should be
"sparse''. We will also need a notion which implies the "density'' of a family contained in $S$. For this purpose, the following classes
have turned out to be convenient.

\begin{defn} \label{mdef0} Write
$S^+=S+\overline{D(0;s/\sqrt{n})},$
where $s$ is some fixed positive number. Let $\mathcal{Z}\subset S$ be a triangular family. We hay that $\mathcal{Z}$ is of class
$M_{S,\rho}$ if $\mathcal{Z}$ is uniformly $2s$-separated and
$$\int_{S^+}\babs{f}^2\le C\frac 1 {n\rho}\sum_{z_{nj}\in \mathcal{Z}_n} \babs{f(z_{nj})}^2,\quad f\in\tilde{H}_{n\rho}$$
for all large $n$.
\end{defn}

\begin{defn} \label{mdef} Let $\delta_n=\log^2 n/\sqrt{n}$ and put
$S_n=\left\{z\in S;~\dist(z,\d S)\ge 2\delta_n\right\}.$
We say that a triangular family $\mathcal{Z}\subset S$ is of class $M_{S_n,\rho}$ if $\mathcal{Z}$ is uniformly separated and
$$\int_{S_n}\babs{f}^2\le C\frac 1 {n\rho}\sum_{z_{nj}\in \mathcal{Z}_n} \babs{f(z_{nj})}^2,\quad f\in\tilde{H}_{n\rho}$$
for all large $n$.
\end{defn}

\subsection{Results in the interior of the droplet} \label{homolu}
We have the following lemma.

\begin{lem} \label{mainr} Let $\zeta=(z_n)$ be a convergent sequence in $S$ with $\dist(z_n,\d S)\ge 3\delta_n$ for all $n$. Then
\begin{enumerate}
\item[(i)] If $\mathcal{Z}$ is of class $M_{S_n,\rho}$, then $D^-(\mathcal{Z};\zeta)\ge \rho$,
\item[(ii)] If $\mathcal{Z}$ is $\rho$-interpolating, then $D^+(\mathcal{Z};\zeta)\le \rho$.
\end{enumerate}
\end{lem}

A proof is given in Section \ref{pmainr}.

When $\mathcal{F}_n$ is a Fekete set, we write $\mathcal{F}_n^\prime=\mathcal{F}_n\cap S_n$ and $\mathcal{F}^\prime=\{\mathcal{F}_n^\prime\}$.

\begin{lem} \label{mt2} One has that
\begin{enumerate}
\item $\mathcal{F}$ is uniformly separated,
\item $\mathcal{F}^\prime$ is $\rho$-interpolating for any $\rho>1$,
\item $\mathcal{F}$ is of class $M_{S_n,\rho}$ whenever $\rho<1$.
\end{enumerate}
\end{lem}

A proof is given in Section \ref{mt2proof}.

Using lemmas \ref{mainr} and \ref{mt2}, we infer that for $\zeta=(z_n)$ with $\dist(z_n,\d S)\ge 3\delta_n$, we have for any $\eps>0$ that
$D^-(\mathcal{F};\zeta)\ge 1-\eps$ and $D^+(\mathcal{F};\zeta)\le 1+\eps$. This finishes the proof
of Theorem \ref{MTH1}, since evidently $D^-\le D^+$. q.e.d.

\subsection{The Ginibre case} \label{gincas} Now let $Q=\babs{z}^2$ so that $S=\overline{\D}$,
and fix a convergent sequence $\zeta=(z_j)$ in $\overline{\D}$ such that the limit
$L=\lim_{n\to\infty}\sqrt{n}(1-\babs{z_n})$ exists.

\begin{lem} \label{h11} Suppose that $Q=\babs{z}^2$, and let $\mathcal{Z}$ be a triangular family contained in $\overline{\D}$.
\begin{enumerate}
\item If $\mathcal{Z}$ is of class $M_{\overline{\D},\rho}$, then
\begin{equation*}
D^-(\mathcal{Z};\zeta)\ge \begin{cases} \rho & \text{if}\quad L=+\infty,\cr
\rho/2 & \text{if}\quad L<+\infty.\cr
\end{cases}
\end{equation*}
\item If $\mathcal{Z}$ is $\rho$-interpolating, then
\begin{equation*}
D^+(\mathcal{Z};\zeta)\le \begin{cases} \rho & \text{if}\quad L=+\infty,\cr
\rho/2 & \text{if}\quad L<+\infty.\cr
\end{cases}
\end{equation*}
\end{enumerate}
\end{lem}

A proof is given in \textsection \ref{hereitis}.

\begin{lem} \label{h12} Let $\mathcal{F}=\{\mathcal{F}_n\}$ be a family of Fekete sets with respect to the potential $Q=\babs{z}^2$.
Then $\mathcal{F}$ is of class $M_{\overline{\D},\rho}$ for any $\rho<1$ and $\rho$-interpolating for any $\rho>1$.
\end{lem}

A proof is given in \textsection\ref{blala}.

To finish the proof of Theorem \ref{MTH2} it suffices to combine Lemma \ref{h11} and Lemma \ref{h12}. q.e.d.

\subsection{Auxiliary lemmas} We state a couple of known facts which are used frequently in the following.
The following uniform estimate is well-known (see e.g. \cite{ST}).

\begin{lem} \label{BAHM} Let $f\in \tilde{H}_n$ and $z\in \C\setminus S$. Assume that $\babs{f}\le 1$ on $S$. Then
$\babs{f(z)}\le \e^{-n\lpar Q(z)-\wh{Q}(z)\rpar/2},\quad z\in \C.$
\end{lem}

(Proof: Let $f=p\cdot \e^{-nQ/2}$. The assumption gives that $\frac 1 n \log \babs{p}^2$ is a subharmonic minorant of $Q$
which grows no faster than $\log\babs{z}^2+\text{const.}$ as $z\to \infty$. Thus $\frac 1 n \log\babs{p}^2\le \wh{Q}$.)

We will also use the following well-known pointwise-$L^2$ estimate.

\begin{lem} \label{subh} Let $f=u\e^{-nQ/2}$ where $u$ is holomorphic and bounded in $D(z_0;c/\sqrt{n})$ for some $c>0$.
Suppose that $\Delta Q(z)\le K$ for all
$z\in D(z_0;c/\sqrt{n})$.
Then
\begin{equation}\label{voms}\babs{f(z_0)}^2\le n\cdot \e^{Kc^2}c^{-2}\int_{D(z_0;c/\sqrt{n})}\babs{f}^2 \dA.\end{equation}
In particular, if $\mathcal{Z}$ is $2s$-separated, then for all $f\in\tilde{H}_n$
\begin{equation}\label{oms}\frac 1 {n} \sum_{z_{nj}\in \Omega}\babs{f(z_{nj})}^2\le C s^{-2}\int_{\Omega^+}\babs{f(\zeta)}^2\dA(\zeta),\end{equation}
where $C$ depends only on the upper bound of $\Delta Q$ on $S^+$ and
$\Omega^+=\left\{\zeta\in\C;~\dist(\zeta,\Omega)\le s/\sqrt{n}\right\}.$
\end{lem}

A proof of \eqref{voms} can be found e.g. in \cite{AHM}, Section 3. The estimate \eqref{oms} is immediate from this.

We will also need the following lemma on uniform estimates and "off-diagonal damping'' for correlation kernels.

\begin{lem}\label{cor8.2} (i) There is a constant $C$ such that for all $z,w\in \C$,
\begin{equation*}\babs{\mathbf{K}_{n}(z,w)}\le Cn\e^{-n(Q(z)-\wh{Q}(z))/2}
\e^{-n(Q(w)-\wh{Q}(w))/2}.\end{equation*}
(ii) Suppose that $z\in S$ and let $\delta=\dist(z,\d S)$. There are then positive constants
$C$ and $c$ such that
$$\babs{\mathbf{K}_n(z,w)}\le Cn\exp\lpar -c\sqrt{n}\min\{\babs{z-w},\delta\}\rpar\cdot \e^{-n(Q(w)-\wh{Q}(w))/2},\quad
w\in \C.$$
\end{lem}

Part (i) is standard, see e.g. \cite{AHM}, Sect. 3.
For a proof of (ii) we refer to \cite{AHM}, Corollary 8.2 (which also shows that the constant $c$ can be taken proportional to
$\inf\{\sqrt{\Delta Q(z)};~z\in S\}$).

\subsection{Notation} We use the same letter $\mathbf{K}$ to denote a kernel $\mathbf{K}(z,w)$ and its corresponding integral operator $\mathbf{K}(f)(z)=\int_\C f(w)\mathbf{K}(z,w)\dA(w)$. We will denote by the same symbol "$C$'' a constant independent of $n$, which can change meaning as we go along. The notation "$A_n\precsim B_n$'' means that $A_n\le CB_n$.
We shall write
\begin{equation}\label{an's}A_n(z)=D(z;R/\sqrt{n})\quad ,\quad A_n^+(z)=D(z;(R+s)/\sqrt{n})\quad ,\quad
A_n^-(z)=D(z;(R-s)/\sqrt{n}).\end{equation}

\section{Preliminary estimates}\label{unisep}

In this section, we discuss gradient estimates for weighted polynomials; these will be useful in the following. In particular they
imply that interpolating families are uniformly separated.

\subsection{Inequalities of Bernstein type} The following lemma is analogous to Lemma 18 in \cite{MMO}.

\begin{lem}\label{bernstein} Let $p$ be a polynomial of degree at most $n$. Fix a point $z$ such that
$p(z)\ne 0$ and $\babs{\Delta Q(z)}< K$. Then
\begin{equation}\label{babel}\babs{\nabla \lpar \babs{p}\e^{-nQ/2}\rpar(z)}\le C\sqrt{n}\left\|p\e^{-nQ/2}\right\|_{L^\infty},\end{equation}
and
\begin{equation}\label{abel}\babs{\nabla\lpar \babs{p}\e^{-nQ/2}\rpar(z)}\le C n\left\|p\e^{-nQ/2}\right\|_{L^2},\end{equation}
where the constant $C$ depends only on $K$.
\end{lem}

\begin{proof} Let
$H_z(\zeta)=Q(z)+2\d Q(z)\cdot (\zeta-z)+\d^2 Q(z)(\zeta-z)^2$ and $h_z(\zeta)=\re H_z(\zeta),$
so that
$$Q(\zeta)=h_z(\zeta)+\Delta Q(z) \babs{\zeta-z}^2+O(\babs{z-\zeta}^3).$$
In particular, there is a constant $C$ such that
\begin{equation}\label{ont}n\babs{Q(\zeta)-h_z(\zeta)}\le C\quad \text{when}\quad \babs{\zeta-z}\le 1/\sqrt{n},\end{equation}
where $C$ depends only on $K$.

Now observe that,
\begin{equation}\label{dont}\babs{\nabla\lpar \babs{p}\e^{-nQ/2}\rpar(\zeta)}=\babs{p^\prime(\zeta)-n\cdot \d Q(\zeta)\cdot p(\zeta)}\e^{-nQ(\zeta)/2},\end{equation}
and
\begin{equation}\label{font}\begin{split}\babs{\nabla\lpar \babs{p}\e^{-nh_z/2}\rpar(\zeta)}&=\babs{p^\prime(\zeta)-n\cdot \d h_z(\zeta)
\cdot p(\zeta)}\e^{-nh_z(\zeta)/2}=\\
&=\babs{\frac \diff {\diff \zeta}(p\e^{-nH_z/2})(\zeta)}.\\
\end{split}
\end{equation}
The expressions \eqref{dont} and \eqref{font} are identical when $\zeta=z$.

By Cauchy's estimate applied to the circle $C_{1/\sqrt{n}}(z)$ with center $z$ and radius $1/\sqrt{n}$,
\begin{equation}\babs{\frac \diff {\diff \zeta}(p\e^{-nH_z/2})(z)}=\frac 1 {2\pi}\babs{\int_{C_{1/\sqrt{n}}(z)}
\frac {p(\zeta)\e^{-nH_z(\zeta)/2}} {(z-\zeta)^2} \diff \zeta}\le \frac n {2\pi}
\int_{C_{1/\sqrt{n}}(z)}\babs{p(\zeta)}\e^{-nh_z(\zeta)/2}\babs{\diff \zeta}.\end{equation}
In view of \eqref{ont}, the right side can be estimated by a constant depending only on $K$, times
\begin{equation}\label{sabel}n\int_{C_{1/\sqrt{n}}(z)}\babs{p(\zeta)}\e^{-nQ(\zeta)/2}\babs{\diff \zeta}.\end{equation}

To prove \eqref{babel}, it suffices to notice that \eqref{sabel} can be estimated by $2\pi\sqrt{n}\left\| p\e^{-nQ/2}\right\|_{L^\infty}$.

Next notice that, by Lemma \ref{subh},
$$\babs{p(\zeta)}^2\e^{-nQ(\zeta)}\le C'n\int_{D(\zeta;1/\sqrt{n})}\babs{p(\xi)}^2\e^{-n Q(\xi)}\dA(\xi)\le C'n
\left\|p\e^{-nQ/2}\right\|_{L^2}^2$$
with another constant $C'$ depending only on $K$.
We conclude that
\begin{equation*}\babs{\nabla\lpar \babs{p}\e^{-nQ/2}\rpar(\zeta)}\precsim n\sqrt{n}\left\|p\e^{-nQ/2}\right\|_{L^2}
\int_{C_{1/\sqrt{n}}(z)}\babs{\diff\zeta}\precsim n\left\|p\e^{-nQ/2}\right\|_{L^2},\end{equation*}
with a constant depending only on $K$. This proves \eqref{abel}.
\end{proof}

\subsection{Proof of Lemma \ref{BASE0}}\label{BASE1} Let $\mathcal{Z}$ be an interpolating family contained in $S$. (W.l.o.g. put $\rho=1$.)

Fix an index $j$, $1\le j\le m_n$.
Since $\mathcal{Z}$ is interpolating, we can find a function $f=f_n\in\tilde{H}_{n}$ such that $f(z_{nj'})=\delta_{jj'}$ and
$\|f\|^2\le C/n$. Let $\delta>0$ be small enough that $D(z_{nj};\delta)\subset \Lambda$.
Also assume w.l.o.g. that a point $z_{nj'}$ satisfies $\babs{z_{nj}-z_{nj'}}<\delta$; if there is no such $j'$ there is nothing to prove.

Evidently,
\begin{equation*}1=\babs{\babs{f(z_{nj})}-\babs{f(z_{nj'})}}\le \left\|\nabla \babs{f}\right\|_{L^\infty(\Lambda)}\babs{z_{nj}-z_{nj'}}.\end{equation*}
Thus Lemma \ref{bernstein} gives
$$1\le C_1n\left\|f\right\| \babs{z_{nj}-z_{nj'}}\le CC_1\sqrt{n}\babs{z_{nj}-z_{nj'}}.$$
This proves that $\mathcal{Z}$ is $s$-separated with $s=1/(CC_1)$. $\qed$

\section{The spectrum of the concentration operator} \label{conce}

Let $\Omega$ be a measurable subset of the plane. The \textit{concentration operator} $\mathbf{K}_{n\rho}^\Omega$
is defined by
$$\mathbf{K}_{n\rho}^\Omega(f)(z)=\int_\Omega f(w)\mathbf{K}_{n\rho}(z,w)\dA(w)=\mathbf{K}_{n\rho}(\1_\Omega\cdot f)(z).$$
This is a positive contraction on $\tilde{H}_{n\rho}$.

In this section, we apply a technique which relates the spectrum of the concentration operator to the number of points
in $\Omega\cap\mathcal{Z}_n$ when $\mathcal{Z}$ is either an interpolating family or an M-family; the technique essentially
goes back to Landau's paper \cite{L}. We here follow the strategy in \cite{OP}, in a suitably modified form.

We first turn to $M$-families. We will consider the cases of $M_S$ and of $M_{S_n}$ families separately.

\subsection{$M_{S,\rho}$-families} Fix a point $z\in S$
and
let $\lambda_j^{n\rho}=\lambda_j^{n\rho}(z)$ denote the eigenvalues of
$\mathbf{K}_{n\rho}^{A_n(z)}:\tilde{H}_{n\rho}\to\tilde{H}_{n\rho}$, taken in decreasing order.
Let $\phi_j^{n\rho}$ be corresponding normalized eigenvectors.
We write
$$
N_{n\rho}^+=N_{n\rho}^+(z)=\#\lpar \mathcal{Z}_n\cap A_n^+(z)\rpar.$$
(See \eqref{an's} for the definitions of the sets $A_n$ and $A_n^+$.)
\begin{lem} \label{sepo} Suppose that $\mathcal{Z}\subset S$ is of class
$M_{S,\rho}$.
There is then a constant $\gamma<1$ and a number $n_0$ such that for all $z\in S$ and $n\ge n_0$, we have
$\lambda_{N_{n\rho}^+(z)+1}^{n\rho}< \gamma.$
\end{lem}

\begin{proof} W.l.o.g. put $\rho=1$. Fix $z\in S$ and suppose that $f\in\tilde{H}_n$ is such that $f(z_{nj})=0$ when $z_{nj}\in A_n^+(z)$. Since $\mathcal{Z}$ is $2s$-separated,
\begin{equation}\label{const3p}\int_{S^+}\babs{f}^2\le C\frac 1 n \sum_{z_{nj}\in S\setminus A_n^+(z)} \babs{f(z_{nj})}^2\le Cs^{-2}
\int_{S^+\setminus A_n(z)}\babs{f}^2.\end{equation}

Now define
$f=\sum_{j=1}^{N_{n}^++1}c_j^n\phi_j^n,$
where the numbers $c_j^n$ (not all zero) are chosen so that $f(z_{nj})=0$ for all $z_{nj}\in A_n^+(z)$. This is possible
for all large $n$, because the separation of $\mathcal{Z}$ ensures that $N_n^+\le C$ for some constant $C=C(R,s)$.

Since the operator $\mathbf{K}_{n}$ is the orthogonal projection of $L^2$
onto $\tilde{H}_{n}$, we have
\begin{equation}\begin{split}\label{kna}\sum_{j=1}^{N_{n}^++1}\lambda_j^n\babs{c_j^n}^2&=
\left\langle \mathbf{K}_{n}^{A_n(z)}f,f\right\rangle=\left\langle \1_{A_n(z)}\cdot\mathbf{K}_{n}(f), \mathbf{K}_{n}(f)\right\rangle
=\int_{A_n(z)}\babs{f}^2\diff A.\\
\end{split}\end{equation}
We infer by means of \eqref{const3p} and \eqref{kna} that
\begin{equation*}\begin{split}\lambda_{N_{n}^++1}^n\sum_{j=1}^{N_{n}^++1}
\babs{c_j^n}^2&\le  \sum_{j=1}^{N_n^++1}\lambda_j^n\babs{c_j^n}^2=
\lpar \int_{S^+}-\int_{S^+ \setminus A_n(z)}\rpar \babs{f}^2\dA\\
&\le\lpar 1-\frac {s^2} {C}\rpar\int_{S^+}\babs{f}^2
\le \lpar 1-\frac {s^2} {C}\rpar\left\|f\right\|^2\le
\lpar 1-\frac {s^2} {C}\rpar\sum_{j=1}^{N_{n}^++1}\babs{c_j^n}^2,\\
\end{split}
\end{equation*}
which proves that $\lambda_{N_n^++1}\le 1-s^2/C$.
\end{proof}

Notice that the separability of $\mathcal{Z}$ implies that  $\#(\mathcal{Z}_n\cap(D(z;(R+s)/\sqrt{n}) \setminus D(z;R/\sqrt{n})))\le CR$. Therefore Lemma \ref{sepo} implies the estimate
\begin{equation}\label{sep2p}\#(\mathcal{Z}_n\cap D(z;R/\sqrt{n}))\ge \#\{j;~\lambda_j^{n\rho}\ge \gamma\}+O(R),\end{equation}
where the $O$-constant is independent of $n$.

\subsection{$M_{S_n,\rho}$-families} We now modify the construction in the previous subsection.

Recall that $S_n=\{z\in S;~\dist(z,\d S)\ge 2\delta_n\}$ where $\delta_n=\log^2 n/\sqrt{n}$.
Fix a sequence $z_n\in S_n$ satisfying
\begin{equation*}\dist(z_n,\d S)\ge 3\delta_n.\end{equation*}
 As before, we consider the eigenvalues $\lambda_j^{n\rho}$
 (decreasing order) and corresponding eigenfunctions $\phi_j^{n\rho}$ of the concentration operator $\mathbf{K}_{n\rho}^{A_n(z_n)}$.
We will use the following lemma.

\begin{lem}\label{chelp} For any positive integer $K$ there is a constant $C_K$ and a number $n_0=n_0(R)$ such that for all $j$
$$\lambda_j^{n\rho}\int_{\C\setminus S_n}\babs{\phi_j^{n\rho}}^2\le C_Kn^{-K},\quad n\ge n_0.$$
\end{lem}

\begin{proof} W.l.o.g. let $\rho=1$. Choose $n_0$ such that $\dist(A_n(z_n),\C\setminus S_n)\ge \delta_n/2$ when $n\ge n_0$.
By Lemma \ref{cor8.2}, we then have an estimate
$$\babs{\mathbf{K}_{n}(\zeta,w)}\le Cn\e^{-c\log^2 n}\e^{-n(Q(w)-\wh{Q}(w))/2},\quad w\in \C\setminus S_n,\quad
\zeta\in A_n(z_n),$$
where $c$ and $C$ are positive constants.
This gives
\begin{equation*}\begin{split}\babs{\int_{\C\setminus S_n}\int_{A_n(z_n)} \phi_j(w)\mathbf{K}_n(\zeta,w)\overline{\phi_j(\zeta)}\dA(\zeta)\dA(w)}&\le Cn\e^{-c\log^2 n}\lpar \int
\babs{\phi_j(z)}\e^{-n(Q(z)-\wh{Q}(z))/2}\rpar^2\le\\
&\le  C_Kn^{-K}\|\phi_j\|^2=C_Kn^{-K},\\\end{split}\end{equation*}
where we have used the Cauchy-Schwarz inequality and that $\int \e^{-n(Q-\wh{Q})}=1+o(1)$.
\end{proof}

\begin{lem} \label{sep} Suppose that $\mathcal{Z}\subset S$ is of class
$M_{S_n,\rho}$.
There is then a constant $\gamma<1$ and a number $n_0=n_0(R)$ such that for all $z\in S$ satisfying $\dist(z,\d S)\ge 3\delta_n$ and all $n\ge n_0$, we have
$\lambda_{N_{n\rho}^+(z)+1}^{n\rho}< \gamma.$
\end{lem}

\begin{proof} W.l.o.g. put $\rho=1$.
Assume that $f\in\tilde{H}_{n}$ is such that $f(z_{nj})=0$ for all $z_{nj}\in \mathcal{Z}_n\cap A_n^+(z)$.
Then since $\mathcal{Z}$ is $2s$-separated for a sufficiently small $s$,
\begin{equation}\label{const3}\begin{split}\int_{S_n} \babs{f}^2 &\le C\frac 1 {n} \sum_{z_{nj}\in S\setminus A_n^+(z)}
\babs{f(z_{nj})}^2\le Cs^{-2}\int_{S^{+} \setminus A_n(z)}\babs{f}^2\dA.\\
\end{split}
\end{equation}
We again define
$f=\sum_{j=1}^{N_{n}^++1}c_j^n\phi_j^n,$
where the numbers $c_j^n$ (not all zero) are chosen so that $f(z_{nj})=0$ for all $z_{nj}\in \mathcal{Z}_n\cap A_n^+(z)$.

This time, observe that Lemma \ref{chelp} and the Cauchy-Schwarz inequality implies
\begin{equation*}\begin{split}\lambda^n_{N_n^++1}&\int_{\C\setminus S_n}\babs{f}^2\le
\sum_{j,k=1}^{N_n^++1}\sqrt{\lambda^n_j\lambda^n_k}\int_{\C\setminus S_n}\babs{c_j^n\phi_j^n\cdot c_k^n\phi_k^n}\le\\
&\le C_Kn^{-K}\sum_{j,k=1}^{N_n^++1}\babs{c_j^nc_k^n}\le
C_Kn^{-K}(N_n^++1)\sum_{j=1}^{N_n^++1}\babs{c_j^n}^2\le C^\prime n^{-k}\sum_{j=1}^{N_n^++1}\babs{c_j^n}^2.\\
\end{split}\end{equation*}

In view of \eqref{const3}, we now conclude that
\begin{equation*}\begin{split}\lambda_{N_{n}^++1}^n\sum_{j=1}^{N_{n}^++1}
\babs{c_j^n}^2&\le  \sum_{j=1}^{N_n^++1}\lambda_j^n\babs{c_j^n}^2=
\lpar \int_{S^+}-\int_{S^+ \setminus A_n(z)}\rpar \babs{f}^2\dA\\
&\le\lpar 1-\frac {s^2} {C}\rpar\int_{S_n}\babs{f}^2 +\int_{S^+\setminus S_n}\babs{f}^2
\le \lpar 1-\frac {s^2} {C}\rpar\left\|f\right\|^2+\int_{S^+\setminus S_n}\babs{f}^2\le \\
&\le
\lpar 1-\frac {s^2} {C}+\frac {C_Kn^{-K}} {\lambda^n_{N_n^++1}}\rpar\sum_{j=1}^{N_{n}^++1}\babs{c_j^n}^2,\\
\end{split}
\end{equation*}
where $C_K$ depends only on $K$, $R$, and $s$.

With $\alpha=1-s^2/C$, this implies
$\lambda^n_{N_n^++1}<\sqrt{\alpha^2+4C_Kn^{-K}}.$
Thus if we define $\gamma$ as any number in the interval $(\alpha,1)$, we obtain $\lambda_{N_n^++1}^n\le \gamma$ for all $n$ large enough.
\end{proof}

As a corollary, we obtain the following estimate: Let $\mathcal{Z}$ be as in Lemma \ref{sep}. Then for all $n\ge n_0(R)$
\begin{equation}\label{sep2}\#(\mathcal{Z}_n\cap D(z;R/\sqrt{n}))\ge \#\{j;~\lambda_j^{n\rho}\ge \gamma\}+O(R),\end{equation}
where the $O$-constant is independent of $n$.

\subsection{Interpolating families} Assume that $\mathcal{Z}$ be a $\rho$-interpolating sequence contained in  $S$, and let $2s$ be a separation constant for $\mathcal{Z}$. We can w.l.o.g.
assume that $\rho=1$.

Fix $z\in S$.
We define $\mathcal{I}_n$ as the set of indices $j$ such that $z_{nj}\in A_n^-(z)$ and let $N_n^-=N_n^-(z)$ be the cardinality of $\mathcal{I}_n$. By the separation we have a uniform bound $N_n^-\le C=C(R,s)$.

Now let $\{c_j\}_1^{m_n}$ be a sequence with $c_j=0$ when $j\not \in \mathcal{I}_n$.
Since $\mathcal{Z}$ is interpolating we can choose $f_{nj}\in \tilde{H}_n$ such that $f_{nj}(z_{nj'})=\delta_{jj'}$ and
 $\left\|f_{nj}\right\|^2\le C/n$ for all $n$ and $j$. The functions $f_{nj}$, $j\in \mathcal{I}_n$ are linearly independent and span an
$N_n^-$-dimensional subspace of $\tilde{H}_n$. We denote this subspace by
$$F=\text{span}\left\{f_{nj};~j\in \mathcal{I}_n\right\}.$$
Note that an arbitrary $f=\sum_{j\in\mathcal{I}_n}c_jf_{nj}\in F$ satisfies
\begin{equation*}\left\|f\right\|^2\le CN_n^-\frac 1 n \sum_{j\in \mathcal{I}_n}\babs{c_j}^2\le C'\frac 1 n \sum_{j\in \mathcal{I}_n}\babs{f(z_{nj})}^2.\end{equation*}
Applying \eqref{oms} now gives
\begin{equation*}\left\|f\right\|^2\le C\int_{A_n(z)}\babs{f}^2=C\left\langle \mathbf{K}_n^{A_n(z)}f,f\right\rangle.\end{equation*}
With $\delta=1/C$, we have shown that
\begin{equation*}\frac {\langle \mathbf{K}_n^{A_n(z)}f,f\rangle} {\langle f,f\rangle}\ge
\delta,\quad f\in F,~f\not\equiv 0.\end{equation*}

Let $\lambda_j^n$ be the eigenvalues of the operator $\mathbf{K}_n^{A_n(z)}$ on $\tilde{H}_n$ arranged in decreasing order.
By the Weyl--Courant lemma (see \cite{DS}, p. 908) we have
\begin{equation*}\lambda_{j-1}^n\ge \inf_{g\in E_j}\frac {\langle \mathbf{K}_n^{A_n(z)}g,g\rangle} {\langle g,g\rangle},\end{equation*}
where $E_j$ ranges over all $j$-dimensional subspaces of $\tilde{H}_n$. Since $\dim{F}=N_n^-$, we obtain
$\lambda_{N_n^--1}^n\ge \delta.$
The construction can obviously be carried out for $\rho\ne 1$ as well.
We have proved the following lemma.

\begin{lem} \label{numb} Suppose that $\mathcal{Z}$ is $\rho$-interpolating, and let $\lambda_j^{n\rho}$ be the
eigenvalues of the operator $\mathbf{K}_{n\rho}^{A_n(z)}$ on $\tilde{H}_{n\rho}$, where $z\in S$. Also let
$N_{n\rho}^-$ be the number of points in $\mathcal{Z}_n\cap A_n^-(z)$. Then
there is a number $\delta>0$ independent of $n$ and $z$ such that
$$\#\{j;~\lambda_j^{n\rho}\ge \delta\}\ge N_{n\rho}^--1.$$
\end{lem}

Next notice that
since $\mathcal{Z}$ is $2s$-separated (Lemma \ref{BASE0}), there is a constant $C$ such that
\begin{equation*}\#(\mathcal{Z}_n\cap D(z;R/\sqrt{n}))-N_{n\rho}^-\le C(R^2-(R-s)^2)/s^2.\end{equation*}
Using Lemma \ref{numb}, we conclude that
\begin{equation}\label{split}\#(\mathcal{Z}_n\cap D(z;R/\sqrt{n}))\le O(R)+\#\{j;~\lambda_j^{n\rho}\ge \delta\},\qquad \text{as}\quad R\to\infty,\end{equation}
where the $O$-constant is independent of $n$.

\section{Beurling--Landau densities of $M$ families and of interpolating families} \label{pmainr}

In this section, we prove Lemma \ref{mainr}. Our proof depends partly on trace estimates
for the concentration operator, which are proved in Section \ref{LATER}.

\subsection{Proof of Lemma \ref{mainr}(i)} \label{p1}
Let $\mathcal{Z}\subset S$ be of class $M_{S_n,\rho}$, and let $\zeta=(z_n)$ be a sequence with $\dist(z_n,\C\setminus S)\ge 3\delta_n$.

Consider the eigenvalues $\lambda_j^{n\rho}=\lambda_j^{n\rho}(z_n)$ of the concentration operator $\mathbf{K}_{n\rho}^{A_n(z_n)}$, and
put
$\mu_n=\sum_{j=1}^{m_n}\delta_{\lambda_j^{n\rho}}$
where $\delta_z$ is the Dirac measure at $z$. We then have
$$\trace\lpar\mathbf{K}_{n\rho}^{A_n(z_n)}\rpar=\int_0^1 x~\diff\mu_n(x)
\quad ,\quad \trace\lpar\mathbf{K}_n^{A_n(z_n)}\circ\mathbf{K}_n^{A_n(z_n)}\rpar=
\int_0^1 x^2 \diff\mu_n(x).$$
Let $\gamma$ and $n_0$ be given by Lemma \ref{sep}. We then have, for all $n\ge n_0(R)$,
\begin{equation*}\begin{split}&\#\{j; \lambda_j^{n\rho}>\gamma\}=\int_\gamma^1 \diff \mu_n(x)\ge \int_0^1 x~\diff \mu_n(x)-\frac 1 {1-\gamma}
\int_0^1 x(1-x)\diff \mu_n(x)=\\
&=\trace\lpar\mathbf{K}_{n\rho}^{A_n(z_n)}\rpar-\frac 1 {1-\gamma}
\left[ \trace\lpar\mathbf{K}_{n\rho}^{A_n(z_n)}\rpar-\trace
\lpar\mathbf{K}_{n\rho}^{A_n(z_n)}\circ\mathbf{K}_{n\rho}^{A_n(z_n)}\rpar\right].\\
\end{split}
\end{equation*}

By the estimate \eqref{sep2} followed by the trace estimates in lemmas \ref{le1} and \ref{le4}, we now get
\begin{equation*}\begin{split}
&\liminf_{n\to\infty}
\frac {\#(\mathcal{Z}_n\cap D(z_n;R/\sqrt{n}))}
{R^2\Delta Q(z_n)}\ge \liminf_{n\to\infty}\frac {\#\{j; \lambda_j^{n\rho}>\gamma\}+O(R)} {R^2\Delta Q(z_n)}\ge
\\
&\ge \liminf_{n\to \infty}\left[\frac{\trace\lpar\mathbf{K}_{n\rho}^{A_n(z_n)}\rpar}
{R^2\Delta Q(z_n)}-\frac 1 {1-\gamma}\frac {\trace\lpar\mathbf{K}_{n\rho}^{A_n(z_n)}\rpar-\trace
\lpar\mathbf{K}_{n\rho}^{A_n(z_n)}\circ\mathbf{K}_{n\rho}^{A_n(z_n)}\rpar}
{R^2\Delta Q(z_n)}\right]+O(1/R)=\\
&= \rho(R+s)^2/R^2+O(1/R).\\
\end{split}
\end{equation*}
Sending
$R\to\infty$, we obtain $D^-(\mathcal{Z};\zeta)\ge \rho$, and the proof of Lemma \ref{mainr}(i) is finished. \qed

\subsection{Proof of Lemma \ref{mainr}(ii)} \label{p3} Let $\mathcal{Z}$ be a $\rho$-interpolating family and let
$\zeta=(z_n)$ be a sequence with $z_n\in S_n$ for all $n$. Again let $\lambda_j^{n\rho}$ be the eigenvalues
of the concentration operator $\mathbf{K}_{n\rho}^{A_n(z_n)}$.

Let $\mu_n$ be the measure $\mu_n=\sum_{j=1}^{m_n}\delta_{\lambda_j^{n\rho}}$. Then for any $\delta\in (0,1)$
$$\{j; \lambda_j^{n\rho}\ge \delta\}=\int_\delta^1 \diff\mu_n(x)\le \int_0^1 x\diff\mu_n(x)+\frac 1 \delta
\int_0^1 x(1-x)\diff\mu_n(x).$$
In view of the estimate \eqref{split}, we can pick $\delta=\delta(R,s)>0$ so that
\begin{equation*}\#(\mathcal{Z}_n\cap D(z;R/\sqrt{n}))\le O(R)+\trace\lpar\mathbf{K}_{n\rho}^{A_n(z_n)}\rpar+\frac 1 \delta
\left[\trace\lpar\mathbf{K}_{n\rho}^{A_n(z_n)}\rpar-\trace\lpar\mathbf{K}_{n\rho}^{A_n(z_n)}\circ \mathbf{K}_{n\rho}^{A_n(z_n)}\rpar\right].\end{equation*}

For $z_n\in S_n$, the trace estimates in lemmas \ref{le4} and \ref{le1} now imply
\begin{equation*}\begin{split}&\limsup_{n\to\infty}\frac {\#(\mathcal{Z}_n\cap D(z_n;R/\sqrt{n}))}
{R^2\Delta Q(z_n)}\le\\
&\le \limsup_{n\to\infty}\frac {\trace(\mathbf{K}_{n\rho}^{A_n(z_n)})} {R^2\Delta Q(z_n)}+
\frac 1 \delta \limsup_{n\to\infty}\frac {\trace\lpar\mathbf{K}_{n\rho}^{A_n(z_n)}\rpar-\trace\lpar\mathbf{K}_{n\rho}^{A_n(z_n)}\circ \mathbf{K}_{n\rho}^{A_n(z_n)}\rpar}
{R^2\Delta Q(z_n)}+O(1/R)=\rho+O(1/R).\\
\end{split}
\end{equation*}
Letting $R\to\infty$ now shows that $D^+(\mathcal{Z};\zeta)\le \rho$, which finishes the proof of Lemma \ref{mainr}(ii). $\qed$

\section{Equidistribution of the bulk part of a Fekete set} \label{mt2proof}

In this section we prove Lemma \ref{mt2}. The proof is given modulo some estimates for the correlation kernel, which are postponed to the next section.

\subsection{Proof of Lemma \ref{mt2}(1)}
Let $\mathcal{F}_n=\{z_{n1},\ldots,z_{nn}\}$ be a Fekete set and
consider the Lagrange interpolation polynomials
\begin{equation*}
l_{nj}(z)=\prod_{i\ne j}(z-z_{ni})/\prod_{i\ne j}(z_{nj}-z_{ni}).\end{equation*}
To avoid bulky notation, from now on write $z_j:=z_{nj}$ etc.

Now consider the \textit{Leja--Siciak function} corresponding to $\mathcal{F}_n$,
$$\Phi_n(z)=\max\left\{\babs{l_j(z)}^2\e^{nQ(z_j)};~j=1,\ldots,n\right\}.$$
It is known that for all $z\in\C$
\begin{equation}\label{lise}\Phi_n(z)^{1/n}\le \e^{\wh{Q}(z)}\quad \text{and}\quad \Phi_n(z)^{1/n}\to \e^{\wh{Q}(z)},\quad \text{as}\quad n\to\infty.\end{equation}
We refer to \cite{ST}, \textsection III.5, notably eq. (5.3) and Corollary 5.3,
for proofs of these statements.

Let us write
\begin{equation}\label{e71}\ell_j(z)=l_j(z)\e^{-n(Q(z)-Q(z_j))/2},\end{equation}
and notice that \eqref{lise} implies that
\begin{equation}\label{e72}\babs{\ell_j(z)}\le \e^{-n(Q(z)-\wh{Q}(z))/2}.\end{equation}

The following lemma concludes our proof for part (1) of Lemma \ref{mt2}.

\begin{lem}\label{jaa} Let $\mathcal{F}=\{\mathcal{F}_n\}$ be a family of Fekete sets. Then $\mathcal{F}$ is uniformly separated.
\end{lem}

\begin{proof} By \eqref{e72} we have $\|\ell_j\|_{L^\infty}\le 1$ for all $j$.
Hence
Lemma \ref{bernstein} implies that there is a neighbourhood $\Lambda$ of $S$ such that
\begin{equation*}\left\|\nabla\lpar \babs{\ell_j}\rpar\right\|_{L^\infty(\Lambda)}\le C\sqrt{n}\end{equation*}
for some constant $C$ independent of $n$ and $j$.

Fix $z_{nj}\in \mathcal{F}_n$
 and assume that a point $z_{nk}\in \mathcal{F}_n$ is sufficiently close to $z_{nj}$. Then
\begin{equation*}
1=\babs{\babs{\ell_j(z_{nj})}-\babs{\ell_j(z_{nk})}}\le \left\|\nabla \babs{\ell_j}\right\|_{L^\infty(\Lambda)}\babs{z_{nj}-z_{nk}}\le C\sqrt{n}\babs{z_{nj}-z_{nk}}.\end{equation*}
We have shown that $\mathcal{F}$ is uniformly separated with best separation constant $\ge 1/C$.
\end{proof}

\subsection{Proof of Lemma \ref{mt2}(2)}
We now modify the weighted Lagrangian polynomials $\ell_j$ \eqref{e71}, by multiplying by certain "peak polynomials'',  to localize to a small neighbourhood of $z_j$.

Take $\eps>0$ small; consider the corresponding
kernel $\mathbf{K}_{\eps n}(z,w)$, and put
\begin{equation}\label{ljd}L_j(z)=\lpar \frac {\mathbf{K}_{\eps n}(z,z_j)} {\mathbf{K}_{\eps n}(z_j,z_j)}\rpar^2
\cdot\ell_j(z).\end{equation}
These are weighted polynomials of degree $(1+2\eps)n$; evidently
$L_j(z_k)=\delta_{jk}.$
We have the following lemma.

\begin{lem} \label{1nlem} There is a constant $C$ depending on $\eps$ but not on $n$ such that for all $z_j\in \mathcal{F}_n\cap S_n$, we have
\begin{equation*}
\left\|L_j\right\|_{L^1}\le\frac C n.\end{equation*}
\end{lem}

\begin{proof} Fix $z_j\in \mathcal{F}_n\cap S_n$. By \eqref{lise} we have
$$\babs{L_j(z)}\le \frac {\babs{\mathbf{K}_{\eps n}(z,z_j)}^2} {\mathbf{K}_{\eps n}(z_j,z_j)^2}\e^{-n(Q(z)-\wh{Q}(z))/2}.$$
Using the asymptotics in Lemma \ref{l1} and the fact that $\wh{Q}\le Q$
everywhere,
we conclude that
\begin{equation*}\left\|L_j\right\|_{L^1}\le \frac C {(\eps n)^2}\int_\C \babs{\mathbf{K}_{n\eps}(z,z_j)}^2\dA(z)=
\frac C {(\eps n)^2}\mathbf{K}_{n\eps}(z_j,z_j)\le
\frac {C'} {\eps n}.\end{equation*}
\end{proof}

\begin{lem} \label{2nlem} Let
$$F_n(z)=\sum_{z_j\in S_n}\babs{L_j(z)}.$$
There are then constants $C=C(\eps,s)$ and $n_0=n_0(\eps)$ such that $\|F_n\|_{L^\infty}\le C$ when $n\ge n_0$.
\end{lem}

\begin{proof} For $z_j\in S_n$ we have $\mathbf{K}_n(z_j,z_j)\ge cn$ where $c>0$, by Lemma \ref{l1}.
Using Lemma \ref{cor8.2} we find that
$$\babs{L_j(z)}\le CV_j(z),\quad z_j\in S_n,\quad z\in \C,$$
where
\begin{equation}\label{vjjdef}
V_j(z)=\exp\lpar -c\sqrt{n\eps}\min\left\{\babs{z-z_j},\delta_{n}\right\}\rpar,\end{equation}
where $c$ is a positive constant.

Observe that $V_j(z)\le \e^{-c\sqrt{\eps}\log^2 n}\le 1/n$ when $\babs{z-z_j}\ge \delta_{n}$ and $n$ is large enough. This gives
that there is $n_0=n_0(\eps)$ such that
$$F_n(z)\le C\sum_{z_j\in D(z;\delta_{n})}V_j(z)+1,\quad n\ge n_0.$$

Now when $z_j\in D(z;\delta_n)$ we have $V_j(z)=\e^{-c\sqrt{n\eps}\babs{z-z_j}}$. Hence when $\babs{w-z_j}\le s/\sqrt{n}$ we have
$V_j(z)\le C\e^{-c\sqrt{n\eps}\babs{z-w}}$,
where $C=\e^{cs\sqrt{\eps}}$. This gives that
$$V_j(z)\le C\frac n {s^2} \int_{D(z_j;s/\sqrt{n})}\e^{-c\sqrt{n\eps}\babs{z-w}}\dA(w).$$
By the separation, we then obtain that, when $n\ge n_0$,
$$F_n(z)\le 1+C\frac n {s^2} \int_\C \e^{-c\sqrt{n\eps }\babs{z-w}}\dA(w)=1+C\frac 1 {s^2\eps} \int_\C \e^{-c\babs{\zeta}}\dA(\zeta)<\infty.$$
The proof of the lemma is finished.
\end{proof}

The following lemma concludes our proof for part (1) of Lemma \ref{mt2}.

\begin{lem} \label{prec} Let $\mathcal{F}=\{\mathcal{F}_n\}_{n=1}^\infty$ be a sequence of Fekete sets. Then the triangular family $\mathcal{F}^\prime$ given by
$\mathcal{F}_n^\prime=\mathcal{F}_n\cap S_n$ is $(1+2\eps)$-interpolating for any $\eps>0$.
\end{lem}

\begin{proof} Write $\mathcal{F}_n\cap S_n=\{z_{n1},\ldots,z_{nm_n}\}$ and take a sequence $c=(c_j)_{j=1}^{m_n}$. Consider the
operator $T:\C^{m_n}\to L^1+L^\infty$ defined by $T(c)=\sum_j c_jL_j$,
where $L_j$ are given by \eqref{ljd}. In view of Lemma \ref{1nlem}
$$\|T\|_{\ell^1_{m_n}\to L^1}\le \sup \|L_j\|_{L^1}\le C/n,$$
and by Lemma \ref{2nlem},
$$\|T\|_{\ell^\infty_{m_n}\to L^\infty}\le \|F_n\|_{L^\infty}\le C.$$
By the Riesz--Thorin theorem, we conclude that
$$\|T\|_{\ell^2_{m_n}\to L^2}\le C/\sqrt{n}.$$
We have shown that, if $f=T(c)$, then $f\in \tilde{H}_{n(1+2\eps)}$, $f(z_{nj})=c_j$ for all $j\le m_n$ and
$$\int \babs{f}^2\le \frac C n\sum_{j=1}^{m_n}\babs{f(z_{nj})}^2.$$
I.e., $\mathcal{F}^\prime$ is $(1+2\eps)$-interpolating.
\end{proof}

\subsection{Proof of Lemma \ref{mt2}(3)}\label{polish} Let $\mathcal{F}=\{\mathcal{F}_n\}$ where the $\mathcal{F}_n$
are $n$-Fekete sets.
We will prove that $\mathcal{F}$ is of class $M_{S_n,1-2\eps}$ whenever $0<\eps<1/2$ and $\eta>0$ (for the definition of this class, see Definition
\ref{mdef}).

Fix a function $f\in\tilde{H}_{n(1-2\eps)}$ with $\eps>0$ small
and a point
$z\in S_n$, and consider the weighted polynomial
$$g_z(\zeta)=f(\zeta)\cdot \lpar \frac {\mathbf{K}_{n\eps}(\zeta,z)} {\mathbf{K}_{n\eps}(z,z)}\rpar^2\in \tilde{H}_{n}.$$

By Lagrange's interpolation formula,
$$g_z(\zeta)=\sum_{j=1}^{n} g_z(z_j)\ell_j(\zeta),$$
where $\ell_j$ is given by \eqref{e71}. It follows that
\begin{equation*}
f(z)=g_z(z)=\sum_{j=1}^{n} f(z_j)\tilde{L}_j(z)\quad \text{where}\quad \tilde{L}_j(z)=\lpar \frac {\mathbf{K}_{n\eps}(z_j,z)} {\mathbf{K}_{n\eps}(z,z)}\rpar^2\ell_j(z).\end{equation*}
This gives (by \eqref{e72})
\begin{equation*}
\babs{f(z)}\le \sum_{j=1}^{n}\babs{f(z_j)}\frac {\mathbf{B}_{n\eps}(z;z_j)}
{\mathbf{K}_{n\eps}(z,z)}\e^{-n(Q(z)-\wh{Q}(z))/2},\end{equation*}
where $\mathbf{B}_{n\eps}(z;w):=\babs{\mathbf{K}_{n\eps}(z,w)}^2/\mathbf{K}_{n\eps}(z,z)$ is the "Berezin kernel'' (see the next section).

\begin{lem} \label{aple} Suppose that $z_j\in S$ and let
$$\tilde{V}_j(z):=\babs{\tilde{L}_j(z)}=\frac {\mathbf{B}_{n\eps}(z;z_j)}
{\mathbf{K}_{n\eps}(z,z)}.$$
There are then constants $C$ and $n_0=n_0(\eps)$ such that
\begin{equation}\label{1sharp}\left\|\tilde{V}_j\right\|_{L^1(S_n)}\le C/n,\quad n\ge n_0.\end{equation}
\end{lem}

\begin{proof} We shall consider two cases: (i) $\dist(z_j,\d S)\ge \delta_n$ and
(ii) $\dist(z_j,\d S)\le \delta_n$.

In case (i) we use the estimates in eq. \eqref{berezgaus} and Lemma \ref{l1}, to conclude
\begin{equation}\label{near}\int_{D(z_j;\delta_n/2)}\frac {\mathbf{B}_{n\eps}(z;z_j)}
{\mathbf{K}_{n\eps}(z,z)}\dA(z)\precsim \int \e^{-n\eps\Delta Q(z)\babs{z-z_j}^2} \dA(z)\le \frac C n,\end{equation}
because $\Delta Q$ is bounded below by a positive constant on $D(z_j;\delta_n/2)$.

Now let $z_j\in S$ be arbitrary and use the off-diagonal damping in Lemma \ref{cor8.2}
coupled with the asymptotic estimate $\mathbf{K}_n(z,z)\sim n$ for $z\in S_n$ (Lemma \ref{l1}) to conclude that there are $C$ and $c>0$ such that
\begin{equation}\label{thesame}\tilde{V}_j(z)\le CV_j(z),\quad z\in S_n,\end{equation}
where $V_j$ is given in \eqref{vjjdef}.

We conclude that
\begin{equation}\label{off}\int_{S_n\setminus D(z_j;\delta_n/2)}\tilde{V}_j(z)\dA(z)\le C\e^{-\frac 1 2 c\sqrt{\eps}\log^2 n}\le 1/n,\end{equation}
provided that $n$ is large enough.
In case (i), the estimate \eqref{1sharp} follows from \eqref{near} and \eqref{off}; case (ii) is immediate from \eqref{off}.
\end{proof}

By Lemma \ref{2nlem} and the estimate \eqref{thesame}, one immediately deduces the following lemma.

\begin{lem}\label{3nlem} Let
$$\tilde{F}_n(z)=\sum_{j=1}^n\tilde{V}_j(z).$$
There are then constants $C$ and $n_0=n_0(\eps)$ such that
$\|\tilde{F}_n\|_{L^\infty(S_n)}\le C$ when $n\ge n_0$.
\end{lem}

We can now conclude the proof of Lemma \ref{mt2}.

By \eqref{1sharp} and \eqref{e72}, the operator $\tilde{T}:(c_j)_1^{n}\mapsto \sum c_j \tilde{L}_j(z)$ is bounded from $\ell^1_n$ to $L^1(S_n)$, of norm $\le C/n$; by Lemma \ref{3nlem} it is also bounded from
$\ell^\infty_n$ to $L^\infty(S_n)$, of norm $\le C$. By interpolation
it is bounded by $C/\sqrt{n}$ from $\ell^2_n$ to $L^2(S_n)$, i.e., we have
\begin{equation*}
\int_{S_n}\babs{f}^2\le \frac C n \sum_{j=1}^{n}\babs{f(z_j)}^2,\quad f\in \tilde{H}_{n(1-2\eps)}.\end{equation*}

We have shown that the family $\mathcal{F}$ is of class $M_{S_n,1-2\eps}$, which concludes our proof of Lemma \ref{mt2}. $\qed$

\section{Trace estimates for the concentration operator} \label{LATER}

In this section, we fill in the gaps in the hitherto discussion, i.e.
we prove trace formulas for the concentration operator. These follow from estimates for the correlation kernel of a type
which is at this point well-known (see e.g. \cite{B},\cite{A},\cite{AHM},\cite{AHM3}). However, the estimates used here are more elementary,
so it has seemed worthwhile to include a brief account of them.

Fix a point $z$ in a small neighbourhood $\Lambda$ of $S$.
The trace of the concentration operator $\mathbf{K}_{\rho n}^{A_n(z)}:\tilde{H}_{\rho n}\to\tilde{H}_{\rho n}$
is given by
\begin{equation*}\trace\lpar \mathbf{K}_{\rho n}^{A_n(z)}\rpar=\int_{A_n(z)}\mathbf{K}_{\rho n}(\zeta,\zeta)\dA(\zeta),\end{equation*}
while the trace of the composition of this operator with itself is
\begin{equation*}\trace\lpar \mathbf{K}_{\rho n}^{A_n(z)}\circ\mathbf{K}_{\rho n}^{A_n(z)}\rpar=\int_{A_n(z)\times A_n(z)}\babs{\bfK_{\rho n}(\zeta,w)}^2\dA(\zeta)\dA(w).\end{equation*}
Recall that
$S_n=\{z\in S; ~\dist(z,\d S)\ge 2\delta_n\}$ and $\delta_n=\log^2 n/\sqrt{n}.$
We shall prove the following lemmas.

\begin{lem}\label{le1} Let $z\in S_n$. Then, as $n\to\infty$,
$$\trace(\mathbf{K}_{n\rho}^{A_n(z)})=
R^2\rho \Delta Q(z)+O(\eps_n),$$
where $\eps_n=\log^{6}n/\sqrt{n}$.
\end{lem}

\begin{lem}\label{le3} Let $z\in S_n$. There is then a constant $C$ such that
\begin{equation}\label{bjug}\trace\lpar \mathbf{K}_{n\rho}^{A_n(z)}\circ\mathbf{K}_{n\rho}^{A_n(z)}\rpar\ge
R^2\rho \Delta Q(z)\lpar 1-C\eps_n\rpar-CR^4\eps_n+O(R),\end{equation}
as $R\to\infty$, where the $O(R)$ constant is independent of $n$.
\end{lem}

Combining the lemmas, we obtain our main auxiliary result on trace estimates in the bulk.

\begin{lem}\label{le4} Let $z\in S_n$. There is a constant $C$ such that
$$\trace\lpar \mathbf{K}_{n\rho}^{A_n(z)}\rpar-\trace\lpar \mathbf{K}_{n\rho}^{A_n(z)}\circ\mathbf{K}_{n\rho}^{A_n(z)}\rpar
\le C\lpar R^2\rho \Delta Q(z)+R^4\rpar \eps_n +O(R),$$
where the $O(R)$ constant is independent of $n$.
\end{lem}

\subsubsection*{Proofs} Define
$$\psi(z,\zeta)=Q(z)+\d Q(z)\cdot (\zeta-z)
+\frac 1 2 \d^2 Q(z)\cdot (\zeta-z)^2,$$
and put
$$\mathbf{K}_{n\rho}^\#(z,w)=n\rho\Delta Q(z)\e^{n\psi(z,\bar{w})-
n(Q(z)+Q(w))/2}.$$
Our proofs of lemmas \ref{le1}--\ref{le4} uses the following asymptotic formula
for the correlation kernel.

\begin{lem}\label{l1} Suppose that $z\in S_n$.
There is then a positive number $C$ independent of $z$ and
$n$ such that for all $w\in D(z;\delta_n)$,
$$\babs{\mathbf{K}_{n\rho}(z,w)-\mathbf{K}^\#_{n\rho}(z,w)}\le C
\lpar n\rho\rpar^2\delta_n^3.$$
\end{lem}

The statement is a suitably modified version of \cite{AHM3}, Theorem 3.2, but it does not follow immediately because the regularity assumption on $Q$ is relaxed in our situation.
A short proof is given in the appendix. Related bulk expansions for correlation kernels are well known, see
 \cite{B}, \cite{A} and the references given there.

Observe that Lemma \ref{le1} is immediate from Lemma \ref{l1}. It remains to prove Lemma \ref{le3}. To this end,
we can apply arguments from \cite{AHM3}, \textsection 3.2. To avoid unnecessary repetition, we shall be brief.

The \textit{Berezin kernel} rooted at a point $\zeta\in \C$ is given by
\begin{equation*}w\mapsto \mathbf{B}_{\rho n}(\zeta;w):=\frac {\babs{\bfK_{\rho n}(\zeta,w)}^2} {\bfK_{\rho n}(\zeta,\zeta)}.\end{equation*}
Notice that
$\int_\C \mathbf{B}_{\rho n}(\zeta;w)\dA(w)=1,$
and that we can write
\begin{equation*}\trace\lpar \mathbf{K}_{\rho n}^{A_n(z)}\circ\mathbf{K}_{\rho n}^{A_n(z)}\rpar=\int_{A_n(z)}
\mathbf{K}_{\rho n}(\zeta,\zeta)\left[\int_{A_n(z)}\mathbf{B}_{\rho n}(\zeta;w)\dA(w)\right]\dA(\zeta).\end{equation*}

Now consider the "heat kernel''
$$\mathbf{G}_{n\rho}(\zeta;w)=n\rho\Delta Q(\zeta)\e^{-n\rho\Delta Q(\zeta)\cdot \babs{\zeta-w}^2}.$$
Using Lemma \ref{l1}, one easily proves that
\begin{equation}\label{berezgaus}\mathbf{B}_{n\rho}(z;w)=\mathbf{G}_{n\rho}(z;w)\cdot \lpar 1+ O(n\delta_n^3)\rpar+O(n^2\delta_n^3),\quad  w\in D(z;\delta_n).
\end{equation}

Next notice that $w\mapsto \mathbf{G}_{n\rho}(\zeta;w)$ is a probability density on $(\C,\dA)$ and that
$\int_{D(\zeta;R/\sqrt{n})}\mathbf{G}_{n\rho}(\zeta;w)\dA(w)=1-\e^{-R^2\rho\Delta Q(\zeta)}.$
Combining with \eqref{berezgaus} we then have, for $\zeta\in S_n$.
$$\int_{D(\zeta;R/\sqrt{n})}\mathbf{B}_{n\rho}(\zeta;w)\dA(w)\ge \lpar 1-\e^{-R^2\rho\Delta Q(\zeta)}\rpar(1-C\eps_n)+R^2\cdot O(\eps_n).$$

We can now continue to estimate
\begin{equation*}\begin{split}\trace&\lpar\mathbf{K}_{n\rho}^{A_n(z)}\circ\mathbf{K}_{n\rho}^{A_n(z)}\rpar\ge
\int_{A_n(z)}\mathbf{K}_{n\rho}(\zeta,\zeta)\left[\int_{D(\zeta;R/\sqrt{n}-\babs{z-\zeta})} \mathbf{B}_{n\rho}(\zeta;w)\dA(w)\right]\dA(\zeta)=\\
&=\int_{D(z;R/\sqrt{n})}(n\rho\Delta Q(\zeta)+O(n\eps_n))\left[
\int_{D(\zeta;R/\sqrt{n}-\babs{z-\zeta})}\lpar \mathbf{G}_{n\rho}(\zeta;w)\cdot \lpar 1+ O(\eps_n)\rpar+ O(n\eps_n)\rpar \dA(w)\right]\dA(\zeta)=\\
&=\int_{D(z;R/\sqrt{n})} (n\rho\Delta Q(\zeta)+O(n\eps_n))\lpar 1-\e^{-\lpar R-\sqrt{n}\babs{z-\zeta}\rpar^2\rho\Delta Q(z)}+R^2\cdot O(\eps_n)\rpar\dA(\zeta).\\
\end{split}
\end{equation*}
Changing variables by $\omega=\sqrt{n}(\zeta-z)$ and writing $r=\babs{\omega}$, we conclude that the dominating term,
 as $n\to\infty$, in the last integral is
$$\rho\Delta Q(z)\int_0^R 2r\lpar 1-\e^{-(R-r)^2\rho \Delta Q(z)}\rpar\diff r=\rho\Delta Q(z)\lpar R^2+O(R)\rpar.$$
The estimate \eqref{bjug} follows from this. $\qed$

\section{The Ginibre case}

In this section we prove lemmas \ref{h11} and \ref{h12}, and, as a consequence, Theorem \ref{MTH2}. We start with some preliminaries on real analytic potentials.

\subsection{Real analytic potentials} Let $Q$ be real-analytic in some neighbourhood $S$ of the droplet.
Consider the function ("joint intensity $k$-point function'')
$$R_{n\rho}^k(\xi_1,\ldots,\xi_k)=\det(\mathbf{K}_{n\rho}(\xi_i,\xi_j))_{i,j=1}^k.$$
We now fix a convergent sequence $(z_n)_1^\infty$ in $S$ and consider the rescaled functions
$$\wh{R}_{n\rho}^k(\zeta_1,\ldots,\zeta_k)=\frac 1 {(n\rho)^k} R_{n\rho}^k(z_n+\zeta_1/\sqrt{n\rho},\ldots,z_n+\zeta_k/\sqrt{n\rho}).$$
We also define the function (\textit{Ginibre$(\infty)$-correlation kernel})
$$\mathbf{k}(\zeta,\eta)=\e^{\zeta\bar{\eta}-\babs{\zeta}^2/2-\babs{\eta}^2/2}.$$
We will use the following lemma.

\begin{lem} \label{bs1} Assume that $\lim_{n\to\infty}\sqrt{n}\dist(z_n,\d S)=+\infty$. Then
for each $k\ge 1$,
\begin{equation}\label{thle1}\lim_{n\to\infty} \wh{R}_{n\rho}^k(\zeta_1,\ldots,\zeta_k)= \det(\mathbf{k}(\zeta_i,\zeta_j))_{i,j=1}^k\end{equation}
with uniform convergence on compact subsets of $\C^k$.
\end{lem}

\begin{proof} In the case when $z_n$ converges to a point of the interior of $S$, the lemma follows from Lemma \ref{l1}.
The same is true if $z_n$ converges to a point of $\d S$ and $\dist(z_n,\d S)\ge \log^2 n/\sqrt{n}$. In the general
case we can use the sharper asymptotic estimate 
for the correlation kernel corresponding to a real analytic potential from the proof of Lemma 4.4 of \cite{AHM3} (or Theorem 2.1 of \cite{A}), together with the rescaling argument in \cite{AHM2}, \textsection 7.5.
Details are omitted.
\end{proof}

\subsection{Trace estimates near the boundary: Ginibre case} \label{hereis} We now specialize to $Q(z)=\babs{z}^2$.
We shall write
$$s_n(\zeta)=\sum_{k=0}^{n-1}\frac {\zeta^k} {k!}.$$
It is well known that
$\mathbf{K}_n(z,w)=
ns_n(nz\bar{w})\e^{-n\frac {\babs{z}^2+\babs{w}^2} 2}.$
We fix a point $z_0\in\T$ and consider
the function
$$\wt{R}_{n\rho}^k(\zeta_1,\ldots,\zeta_k)=\frac 1 {(n\rho)^k}R_{n\rho}^k(z_0(1+\zeta_1/\sqrt{n\rho}),\ldots,z_0(1+\zeta_k/\sqrt{n\rho})).$$
We will denote by
$$\Phi(z)=\frac 1 2 \erfc\lpar -\frac z {\sqrt{2}}\rpar.$$
This is the analytic continuation to $\C$ of the d.f. of a standard normal random variable. Also let
$$\mathbf{F}(z,w)=\e^{z\bar{w}-\babs{z}^2/2-\babs{w}^2/2}\Phi(-z-\bar{w}).$$

\begin{lem} \label{bs2} Suppose that $Q=\babs{z}^2$ and that $z_0\in\T$. Then for each $k\ge 1$,
\begin{equation}\label{thle3}\lim_{n\to\infty}\wt{R}_{n\rho}^k(\zeta_1,\ldots,\zeta_k)=\det\lpar \mathbf{F}(\zeta_i,\zeta_j)\rpar_{i,j=1}^k\end{equation}
with uniform convergence on compact subsets of $\C^k$.
\end{lem}

A proof is given in  \cite{BS}, Theorem C.1(2).

\begin{rem}\label{fm1} Since $\mathbf{F}$ is the correlation kernel of a det-process, we have
$\babs{\mathbf{F}(z,w)}^2\le \mathbf{F}(z,z)\mathbf{F}(w,w)<1.$
\end{rem}

\begin{rem} \label{lowb} Let $\D^+=\D+D(0;s/\sqrt{n})$ where $s>0$. We can then assert that there is a positive constant $c$ such that $\mathbf{K}_n(z,z)\ge cn$ for all
$z\in \D^+$.
(To see this,
let $r=\babs{z}^2$, so that $\mathbf{K}_n(z,z)=nf(nr)$ where $f(t)=\sum_{j=0}^{n-1}\frac {t^j} {j!}\e^{-t}$.
We have $f^\prime(t)=-t^{n-1}\e^{-t}/(n-1)!$ so $f$ is decreasing on $[0,\infty)$. Also, by Lemma \ref{bs2},
$f(n(1+s/\sqrt{n})^2)\to \Phi(-2s)>0$ as $n\to \infty.$)
\end{rem}

By the preceding lemmas, we conclude the following.

\begin{enumerate}
\item[(i)]
Suppose that $\sqrt{n}(1-\babs{z_n})\to +\infty$. Then
\begin{equation}\label{lemiii}\lim_{n\to\infty}\trace(\mathbf{K}_{n\rho}^{A_n(z)})=R^2\rho
\quad \text{and}\quad \lim_{n\to\infty}\trace(\mathbf{K}_{n\rho}^{A_n(z)}\circ \mathbf{K}_{n\rho}^{A_n(z)})= R^2\rho+O(R).\end{equation}
\item[(ii)]
If $\sqrt{n}(1-\babs{z_n})\to L<+\infty$, then
$$\lim_{n\to\infty}\trace(\mathbf{K}_{n\rho}^{A_n(z)})=\rho\int_{D(L;R)}\Phi(-2\sqrt{\rho}\re \zeta)\dA(\zeta),$$
$$\lim_{n\to\infty}\trace(\mathbf{K}_{n\rho}^{A_n(z)}\circ \mathbf{K}_{n\rho}^{A_n(z)})=\rho^2\iint_{D(L;R)\times D(L;R)}
\e^{-\rho\babs{\zeta-\eta}^2}\babs{\Phi\lpar-\sqrt{\rho}(\zeta+\bar{\eta})\rpar}^2\dA(\zeta)\dA(\eta).$$
\end{enumerate}

We need to compare the integrals in (ii). To this end, fix a sequence $z=(z_n)\in \overline{\D}$ and suppose that the limit
$L=\lim_{n\to\infty}\sqrt{n}(1-\babs{z_n})$ exists and is finite.
Observe that
\begin{equation*}
\lim_{R\to\infty}\lim_{n\to\infty}\frac {\trace(\mathbf{K}_{n\rho}^{A_n(z)})} {R^2}= \lim_{R\to\infty}
\frac \rho {R^2}\int_{D(L;R)} \Phi(-2\sqrt{\rho}\re \zeta)\dA(\zeta).\end{equation*}
To compute the integral in the right hand side, we apply the change of variables $\omega=(\zeta-L)/R$. It yields
$$\frac \rho {R^2}\int_{D(L;R)} \Phi(-2\sqrt{\rho}\re \zeta)\dA(\zeta)=\rho\int_\D
\Phi\lpar-2\sqrt{\rho}R\omega-2\sqrt{\rho}L\rpar \dA(\omega).$$
As $R\to\infty$ the right hand side above converges to
$\rho\int_\D \1_{\{\re\omega<0\}}\dA(\omega)=\frac \rho 2.$
We have shown that
\begin{equation}\label{bjus1}\lim_{R\to\infty}\lim_{n\to\infty}\frac {\trace(\mathbf{K}_{n\rho}^{A_n(z)})} {R^2}=\frac \rho 2.
\end{equation}

We also need to calculate the trace of the composition of $\mathbf{K}_{n\rho}^{A_n(z)}$ with itself.
For this purpose, it will be convenient to use the \textit{Dawson's function}
\begin{equation*}F(t):=\e^{-t^2}\int_0^t\e^{x^2}\diff x,\quad t\in\R.\end{equation*}

\begin{lem}\label{Dawsonne} For all $z , w \in \C$ holds
\begin{equation}\label{fusk}\babs{\mathbf{F}(z,w)}\le \e^{-\babs{z-w}^2/2}+\e^{-\left[\re(z-w)\right]^{2}/ 2}\frac 1 {\sqrt{\pi}}
F\lpar \frac {\im(z-w)} {\sqrt{2}}\rpar.\end{equation}
In particular, $\mathbf{F}(z,w)\to 0$ as $\babs{z-w}\to\infty$.
\end{lem}

\begin{proof} We have
$\babs{\mathbf{F}(z,w)}=\e^{-\babs{z-w}^{2}/2}\babs{\Phi\lpar-z-\bar{w}\rpar}.$
By Cauchy's theorem we can unambiguously write
$\Phi(z)=\frac 1 {\sqrt{2\pi}}\int_{-\infty}^z\e^{-\zeta^2/2}\diff \zeta.$
Let $-z-\bar{w}=a+\imag b$ with $a=-\re(z+w)$ and $b=\im(w-z)$, and put $\Gamma=(-\infty,a]\cup [a,a+\imag b]$. An obvious estimate of the integral over $\Gamma$ gives
\begin{equation*}\babs{\Phi\lpar-z-\bar{w}\rpar}\le \Phi(a)+\frac 1 {\sqrt{2\pi}}\int_0^b\e^{t^2/2}\diff t,\end{equation*}
so, since $\Phi(a)\le 1$,
\begin{equation*}
\babs{\mathbf{F}(z,w)}\le \e^{-c^2/2}+\e^{-c^2/2}\frac 1 {\sqrt{2\pi}}\int_0^b\e^{t^2/2}\diff t, \quad c=\babs{z-w}.\end{equation*}
This is equivalent to \eqref{fusk}.

It remains to be shown that $\babs{\mathbf{F}(z,w)}\to 0$ as $\babs{z-w}\to\infty$.
It is well-known that $F$ has asymptotic expansion $F(t)=1/2t+1/4t^3+\cdots$ as
$t\to\infty$ (see \cite{SO}, p. 406), so
$F(t)\to 0$ as $t\to\infty$. Writing $s=\babs{\re(z-w)}$ and $t=\babs{\im(z-w)}$, we notice that it follows from \eqref{fusk} that
\begin{equation*}\babs{z-w}\babs{\mathbf{F}(z,w)}\le \babs{z-w}\e^{-\babs{z-w}^2/2}+
s\e^{-s^2/2}F\lpar t/\sqrt{2}\rpar+\e^{-s^2/2}\cdot tF\lpar t/\sqrt{2}\rpar\le C,\end{equation*}
with a $C$ independent of $z$ and $w$. This finishes the proof, since $\babs{\mathbf{F}}<1$ (Remark \ref{fm1}).
\end{proof}

\begin{lem} \label{lastl} For any $z\in\C$ we have
\begin{equation*}\rho^2\int_\C \e^{-\rho\babs{z-w}^2}\babs{\Phi\lpar -\sqrt{\rho}(z+\bar{w})
\rpar}^2\dA(w)=\rho\Phi(-2\sqrt{\rho}\re z).\end{equation*}
\end{lem}

\begin{proof} W.l.o.g. put $\rho=1$ and
consider the integral
\begin{equation*}I(z)=
\int_\C\e^{-\babs{z-w}^2}
\babs{\Phi(-z-\bar{w})}^2\dA(w).
\end{equation*}
Putting $w=z+r\e^{\imag\theta}$ gives
\begin{equation}\label{i1}I(z)=\frac 1 \pi
\int_0^{2\pi}\left[
\int_0^\infty r\e^{-r^2}
\babs{\Phi\lpar-2\re z-r\e^{-\imag\theta}\rpar}^2
\diff r\right]\diff\theta.\end{equation}
Let $J(z)$ be the inner integral,
$J(z)=\int_0^\infty r\e^{-r^2}
\babs{\Phi\lpar-2\re z-r\e^{-\imag\theta}\rpar}^2
\diff r.$
An integration by parts gives
\begin{equation*}J(z)=\left[-\frac 1 2 \e^{-r^2}\babs{\Phi\lpar-2\re z-r\e^{-\imag\theta}\rpar}^2\right]_{r=0}^\infty+
\frac 1 2\int_0^\infty \e^{-r^2}\frac \d {\d r}\babs{\Phi\lpar-2\re z-r\e^{-\imag\theta}\rpar}^2\diff r.\end{equation*}
By Lemma \ref{Dawsonne}, we have for all $\theta$
 \begin{equation*}\lim_{r\to\infty}\e^{-r^2}\babs{\Phi\lpar-2\re z-r\e^{-\imag\theta}\rpar}^2= 0.
 \end{equation*}
This shows that
\begin{equation}\label{i2}J(z)=\frac 1 2 \Phi(-2\re z)+
\frac 1 2\int_0^\infty \e^{-r^2}\frac \d {\d r}\babs{\Phi\lpar-2\re z-r\e^{-\imag\theta}\rpar}^2\diff r.
\end{equation}
But
\begin{equation}\label{i3}\begin{split}\frac \d {\d r}\babs{\Phi\lpar-2\re z-r\e^{-\imag\theta}\rpar}^2=
-\frac 1 {\imag r}\frac \d {\d\theta}\babs{\Phi\lpar-2\re z-r\e^{-\imag\theta}\rpar}^2 .\\
\end{split}
\end{equation}
Using \eqref{i2} and \eqref{i3} in \eqref{i1} we get
\begin{equation*}I(z)=\Phi(-2\re z)+\frac 1 {2\pi}\int_0^\infty \e^{-r^2}\frac \imag r
\left[
\int_0^{2\pi}\frac \d {\d\theta}\babs{\Phi\lpar-2\re z-r\e^{-\imag\theta}\rpar}^2\diff\theta\right]\diff r.
\end{equation*}
The inner integral in the right hand side clearly vanishes, so
$I(z)=\Phi(-2\re z),$
as desired.
\end{proof}

Suppose now that $L=\lim_{n\to\infty}\sqrt{n}(1-\babs{z_n})<\infty$. By the last lemma, then
\begin{equation}\label{bjus0}\lim_{n\to \infty}\trace(\mathbf{K}_{n\rho}^{A_n(z)}\circ\mathbf{K}_{n\rho}^{A_n(z)})=\rho\int_{D(L;R)}\Phi(-2\re \sqrt{\rho}\zeta)(1+o(1))\dA(\zeta)=
\lim_{n\to\infty}\trace(\mathbf{K}_{n\rho}^{A_n(z)})+o(R^2).\end{equation}

We now have the trace estimates needed for our discussion of Beurling--Landau densities close to the boundary.

\subsection{Proof of Lemma \ref{h11}} \label{hereitis} Let $\mathcal{Z}$ be a triangular family contained in $\overline{\D}$.
Fix a sequence $\zeta=(z_n)$ in $\overline{\D}$ and let $L:=\lim_{n\to\infty}\sqrt{n}(1-\babs{z_n})$.

Assume first that $\mathcal{Z}$ is of class $M_{\D,\rho}$. Using the estimate \eqref{sep2p} and the trace estimates
in the preceding subsection, one can finish the argument exactly as in Section \ref{pmainr} above. More precisely,
if $L=+\infty$ we use the estimates in \eqref{lemiii}
to obtain that $D^-(\mathcal{Z},\zeta)\ge \rho$, and if
$L<\infty$ we use instead the estimates in equations \eqref{bjus1} and
\eqref{bjus0} to obtain $D^-(\mathcal{Z},\zeta)\ge \rho/2$.

Similarly, if $\mathcal{Z}$ is $\rho$-interpolating, we repeat the argument in Section \ref{pmainr} using
the trace estimates above. It yields that $D^+(\mathcal{Z},\zeta)\le \rho$ if $L=\infty$ and $D^+(\mathcal{Z},\zeta)\le \rho/2$
if $L<\infty$.
$\qed$

\subsection{Off-diagonal damping in the Ginibre case} \label{odd} The aim of this subsection is to prove the following proposition.

\begin{prop} \label{zick} Let $Q=\babs{z}^2$ be the Ginibre potential. There exists a constant $C$ such that
\begin{equation}\label{true0}\babs{\mathbf{K}_n(z,w)}\le Cn \frac 1 {1+\sqrt{n}\babs{z-w}},\quad z,w\in \D^+=D(0;1+s/\sqrt{n}).\end{equation}
\end{prop}

We may w.l.o.g. assume that the separation constant $s$ has $s\le 1$, this will be done below.

Our proof consists of checking a number of cases. The argument is somewhat lengthy, but straightforward.

It will facilitate to note that we have rotational symmetry $\mathbf{K}_n(z,w)=\mathbf{K}_n(\e^{\imag\theta}z,\e^{\imag\theta}w)$,
so we may w.l.o.g. assume that $w=r$ is real and non-negative when proving \eqref{true0}.

Also notice that we trivially have
\begin{equation}\label{fofo}\babs{\mathbf{K}_n(z,w)}\le Cn\le C(1+M)n\frac 1 {1+\sqrt{n}\babs{z-w}},\qquad \babs{z-w}\le M/\sqrt{n}.\end{equation}
We shall dispose of another simple case.

\begin{lem}\label{ble} Suppose that $z,r\in \D^+$ and $r\ge 1/2$. Then $\babs{z-r}\le 2\babs{zr-1}+4s/\sqrt{n}$.
\end{lem}

\begin{proof} First assume that $z,r\in \D$. Let $z=x+\imag y$. The inequalities
$$(x-r)^2\le (rx-1)^2,\quad (-1\le x\le 1)\quad \text{and}\quad  y^2\le 4(ry)^2$$ show that
$\babs{z-r}^2\le (rx-1)^2+4(ry)^2\le 4\babs{zr-1}^2.$

It is now straightforward to check the cases when $z$ and/or $r$ are in $\D^+\setminus \D$; we omit details.
\end{proof}

The lemma implies that $\babs{\mathbf{K}_n(z,r)}\le Cn \frac 1 {1+\sqrt{n}\babs{z-r}}$ when $z,r\in \D^+$
are such that $r\ge 0$ and $\babs{zr-1}\le M/\sqrt{n}$.

In the following we can thus assume that $\babs{zr-1}\ge M/\sqrt{n}$ where the constant $M$ is at our disposal.
We make the following observation.

\begin{fact} \label{4l} If $M\ge 8s$, then for all $z,r\in \D^+$ such that $r\ge 1/2$ and $\babs{z-r}\ge M/\sqrt{n}$ we have that
$\babs{z-r}\le 4\babs{zr-1}$.
\end{fact}

\begin{proof} The hypothesis gives that $4s/\sqrt{n}\le \babs{z-r}/2$. Hence Lemma \ref{ble} shows that
$\babs{z-r}\le 2\babs{zr-1}+\babs{z-r}/2.$
\end{proof}

\begin{lem}\label{bmlem0} There are constants $\delta>0$ and $C$ such that for all $z,w\in D(0;2)$ such that $M/\sqrt{n}\le \babs{z\bar{w}-1}\le \delta$, we have
$\babs{\mathbf{K}_n(z,w)}\le Cn\lpar \e^{-n\babs{z-w}^2/2}+\frac 1 {1+\sqrt{n}\babs{z\bar{w}-1}}\rpar.$
\end{lem}

\begin{proof}
Let $\phi(\zeta)=\zeta-1-\log\zeta$ for $\zeta$ close to $1$, with the principal branch of the logarithm. It was observed in \cite{BM}, Appendix B, that
 $\phi$ has an analytic square-root $\xi=\sqrt{\phi}$, which is moreover conformal in a neighbourhood of $1$. We
 fix $\xi(z)$ by requiring it to be negative for real $z\in (0,1)$.

We will now apply \cite{BM}, Theorem B.1, which yields that there exists $\delta>0$ such that for any $M>1$ and all
$\zeta$ with $M/\sqrt{n}\le \babs{\zeta-1}\le \delta$, we have the following, partially overlapping, asymptotic expansions:
\begin{equation*}s_n(n\zeta)\e^{-n\zeta}=\begin{cases}1-\frac 1 {2\sqrt{2}\xi^\prime(\zeta)}\erfc(-\sqrt{n}\xi(\zeta))(1+O(1/\sqrt{n})),&\babs{\arg(\zeta-1)-\pi}\le 2\pi/3,\cr
\frac 1 {2\sqrt{2}\xi^\prime(\zeta)}\erfc(\sqrt{n}\xi(\zeta))(1+O(1/\sqrt{n})),&\babs{\arg(\zeta-1)}\le 2\pi/3.\cr
\end{cases}
\end{equation*}
Applying the asymptotic expansion (found in \cite{BM}, eq. (B.36))
 \begin{equation*}\erfc(-z)=-\frac {\e^{-z^2}}{\sqrt{\pi}z}(1+O(1/z^2))\quad \text{as}\quad z\to\infty,\quad \babs{\arg(z-1)-\pi}<3\pi/4-\eps,\end{equation*}
and the fact that $\erfc(-z)+\erfc(z)=2$, we obtain in both cases
$$\babs{s_n(nz\bar{w})\e^{-nz\bar{w}}}\le C\lpar 1+ \frac {\babs{zw}^{n+1}\e^{-n\re(z\bar{w})+n}}
{\sqrt{n}\babs{z\bar{w}-1}}\rpar,\quad M/\sqrt{n}\le \babs{z\bar{w}-1}\le \delta.$$
We have shown that
$$\babs{\mathbf{K}_n(z,w)}=n\babs{s_n(nz\bar{w})\e^{-nz\bar{w}}}\babs{\e^{nz\bar{w}-n\frac {\babs{z}^2+\babs{w}^2}
2}}\le Cn\lpar \e^{-n\babs{z-w}^2/2}+\frac {\babs{zw}} {\sqrt{n}\babs{z\bar{w}-1}} \babs{zw}^n \e^{-n\lpar \frac {\babs{z}^2
+\babs{w}^2} 2-1\rpar}\rpar.$$
Using that $x^n\e^{-n(x-1)}\le 1$ for $x>0$ we find that
$$\babs{\mathbf{K}_n(z,w)}\le Cn\lpar \e^{-n\babs{z-w}^2/2}+ \frac 1 {\sqrt{n}\babs{z\bar{w}-1}}\rpar,\quad M/\sqrt{n}\le \babs{z\bar{w}-1}\le \delta.$$
This finishes the proof of the lemma.
\end{proof}

If $z,r\in \D^+$ and $M/\sqrt{n}\le \babs{zr-1}\le \delta$, then $\babs{z},r\ge 1/2$ provided that $\delta$ is sufficiently small.

Hence Lemma \ref{bmlem0} and Fact \ref{4l} show that, provided $M\ge 8s$, we have the estimate
$$\babs{\mathbf{K}_n(z,r)}\le Cn\lpar \e^{-n\babs{z-r}^2/2}+ \frac 1 {\sqrt{n}\babs{z-r}/4}\rpar.$$
Since $x\e^{-x^2/2}< 1$, it yields that $\babs{\mathbf{K}_n(z,r)}\precsim \sqrt{n}/\babs{z-r}$.

There now only remains to handle the case when $\babs{zr-1}\ge \delta$. To this end, we shall use the following lemma.

\begin{lem} \label{zock10} If $z,w\in D(0;2)$ satisfy $\babs{1-z\bar{w}}\ge \delta>0$  then
$\babs{\mathbf{K}_n(z,w)}\le C\sqrt{n}\lpar 1+ \frac 1 {\babs{z-w}}\rpar$
where $C$ depends only on $\delta$.
\end{lem}

\begin{proof} We shall use some classical asymptotic estimates due to Szeg\H{o}.
Namely, by \cite{Sz}, Hilfssatz 1, it holds that for all $\zeta\in \C$ with $\babs{\zeta-1}\ge \delta$ we have two
partially overlapping possibilities. Viz. there are open sets $\Omega_1,\Omega_2$ with $\Omega_1\cup\Omega_2=\C\setminus \{1\}$ such that
\begin{equation*}s_n(n\zeta)\e^{-n\zeta}=\begin{cases}\frac 1 {\sqrt{2\pi n}}(\zeta\e^{1-\zeta})^n\frac \zeta {1-\zeta}(1+\eps_n^{(1)}(\zeta)),&
\zeta\in \Omega_1,\cr
1+ \frac 1 {\sqrt{2\pi n}}(\zeta\e^{1-\zeta})^n\frac \zeta {1-\zeta}(1+\eps_n^{(2)}(\zeta)),& \zeta\in\Omega_2.\cr
\end{cases}\end{equation*}
where $\eps_n^{(j)}(\zeta)$ denotes a quantity converging to zero, uniformly on compact subsets of $\Omega_j$.

These relations imply
$$\babs{\mathbf{K}_n(z,w)}\le n\lpar\e^{-n\babs{z-w}^2/2}+C \frac 1 {\sqrt{n}\babs{z\bar{w}-1}}\rpar\le n\lpar\e^{-n\babs{z-w}^2/2}+C\delta^{-1} \frac 1 {\sqrt{n}}\rpar,$$
for $(z,w)$ in a compact subset of $\C^2\setminus \{(z,w);~\babs{z\bar{w}-1}\le \delta\}$. The lemma follows.
\end{proof}

The lemma shows that $\babs{\mathbf{K}_n(z,r)}\le C\delta^{-1}\sqrt{n}$ when $z,r\in\D^+$ satisfy $\babs{1-zr}\ge \delta$ and $r\ge0$.
Thereby,
Proposition \ref{zick} is completely proved. $\qed$

\subsection{Proof of Lemma \ref{h12}} \label{blala} Let $\mathcal{F}=\{\mathcal{F}_n\}$ be a family of Fekete sets
corresponding to $Q=\babs{z}^2$. Also fix a small $\eps>0$. We will prove that $\mathcal{F}$ is of class
$M_{\overline{\D},1-3\eps}$ and that $\mathcal{F}$ is $(1+3\eps)$-interpolating.

To this end, write $\mathcal{F}_n=\{z_1,\ldots,z_n\}$ and introduce the auxiliary functions
$$\ell_j(z)=\lpar \prod_{i\ne j}(z-z_i)/\prod_{i\ne j}(z_j-z_i)\rpar \e^{-n(Q(z)-Q(z_j))/2},$$
as well as
$$L_j(z)=\lpar \frac {\mathbf{K}_{\eps n}(z,z_j)} {\mathbf{K}_{\eps n}(z_j,z_j)}\rpar^3 \ell_j(z)\quad,
\quad \tilde{L}_j(z)=\lpar \frac {\mathbf{K}_{\eps n}(z_j,z)} {\mathbf{K}_{\eps n}(z,z)}\rpar^3 \ell_j(z).$$
By general results, we have that $\mathcal{F}\subset \overline{\D}$, $\mathcal{F}$ is uniformly separated,
and $\babs{\ell_j}\le 1$ on $\C$ (see Theorem \ref{eqcon} and Lemma \ref{jaa}). We also recall:

(i) To prove that $\mathcal{F}$ is $(1+3\eps)$-interpolating it suffices to show that the operator $T:(c_j)_{j=1}^n\mapsto
\sum_{j=1}^n c_j L_j(z)$ is bounded from $\ell_2^n$ to $L^2$, of norm at most $C/\sqrt{n}$. (Consider the weighted polynomial
$f=T(c)\in \tilde{H}_{n(1+3\eps)}$.)

(ii) Similarly,
to prove that $\mathcal{F}$ is of class $M_{\overline{\D},1-3\eps}$ it suffices to show that the operator
$\tilde{T}:(c_j)_{j=1}^n\mapsto
\sum_{j=1}^n c_j \tilde{L}_j(z)$ is bounded from $\ell_2^n$ to $L^2(\D^+)$, of norm at most $C/\sqrt{n}$. This follows from the representation $f=\tilde{T}(c)$, $f\in\tilde{H}_{n(1-3\eps)}$, where $c_j=f(z_j)$. (See \textsection \ref{polish}.)

Next observe that since $\mathbf{K}_{n\eps}(z,z)\ge cn\eps$ for $z\in \D^+$ (Remark \ref{lowb}), and since $\babs{\mathbf{K}_n(z_j,z)}\le Cn$, we have
$$\babs{L_j(z)}\le \frac C {n^2\eps^3} {\babs{\mathbf{K}_{n\eps}(z_j,z)}^2},\quad z\in \C$$
and
$$\babs{\tilde{L}_j(z)}\le \frac C {n^2\eps^3} {\babs{\mathbf{K}_{n\eps}(z_j,z)}^2},\quad z\in \D^+.$$
Since
$$\int_\C \babs{\mathbf{K}_{n\eps}(z_j,z)}^2\dA(z)=\mathbf{K}_{\eps n}(z_j,z_j)\le n\eps,$$
we conclude that
$\|L_j\|_{L^1}\le C/(n\eps^2)$ and $\|\tilde{L}_j\|_{L^1}\le C/(n\eps^2)$. We have shown that
$\|T\|_{\ell^1_n\to L^1}\le C/(n\eps^2)$ and $\|\tilde{T}\|_{\ell^1_n\to L^1(\D^+)}\le C/(n\eps^2)$.

Next notice that by Proposition \ref{zick} and Remark \ref{lowb} we have
$$\babs{L_j(z)}\le CV_j(z)\quad \text{and} \quad \babs{\tilde{L}_j(z)}\le CV_j(z),\quad z\in \D^+,$$
where
$$V_j(z)= \frac 1 {(1+\sqrt{n\eps}\babs{z-z_j})^3}.$$
Let us introduce the function
$$F_n(z)=\sum_{j=1}^n V_j(z).$$

Since
$$V_j(z)\le Cns^{-2}\int_{D(z_j,s/\sqrt{n})}
V_j(w)\dA(w),\quad z\in \D^+,$$
with a constant $C$ depending only on $s$ and $\eps$,
the separation implies that
$$F_n(z)\le Cns^{-2}\int_{\D^+}\frac 1 {(1+\sqrt{n\eps}\babs{z-w})^3}\dA(w)\le C\eps^{-1}s^{-2}\int_{\C}\frac 1
{(1+\babs{\zeta})^3}\dA(\zeta)<\infty,\quad z\in \D^+.$$
This implies $\|F_n\|_{L^\infty(\D^+)}\le C\eps^{-1}s^{-2}$, which gives $\|T\|_{\ell^\infty_n\to L^{\infty}(\D^+)}\le
C\eps^{-1}s^{-2}$ and $\|\tilde{T}\|_{\ell^\infty_n\to L^{\infty}(\D^+)}\le
C\eps^{-1}s^{-2}$.

We now recall that the bound $\|T\|_{\ell^\infty_n\to L^{\infty}(\D)}\le C$ means that, for all $(c_j)_1^n$,
$$\left\|\sum_{j=1}^n c_jL_j\right\|_{L^\infty(\D)}\le C\max \babs{c_j}.$$
But by the maximum principle (Lemma \ref{BAHM}), the weighted polynomial $\sum c_jL_j(z)$ assumes its maximum on $\overline{\D}$, which means that
$$\left\|\sum_{j=1}^n c_jL_j\right\|_{L^\infty}\le C\max\babs{c_j}.$$
We have shown that $\|T\|_{\ell^\infty_n\to L^{\infty}}\le C\eps^{-1}s^{-2}$.

By interpolation we now infer that $\|T\|_{\ell^2_n\to L^2}\le C/\sqrt{n}$ and $\|\tilde{T}\|_{\ell^2_n\to L^2(\D^+)}\le C/\sqrt{n}$ with a constant depending on $\eps$ and $s$. The proof of the lemma is finished. q.e.d.

\section*{Appendix: The proof of Lemma \ref{l1}}
Let $Q$ be $C^3$-smooth in some neighbourhood of $S$;
we assume that $\Delta Q\ge \text{const.}>0$ there.
To prove Lemma \ref{l1}, we shall use a simplified form of the argument used in the appendix of \cite{AHM3}.

To simplify the discussion
we put $\rho=1$; for the general case one needs simply to replace "$n$'' by "$n\rho$''.

It will be useful to keep in mind the following elementary properties of the equilibrium potential $\wh{Q}$: (i) $\wh{Q}$ is $C^1$-smooth on $\C$
and the gradient of $\wh{Q}$ is Lipschitz continuous on $\C$, (ii) $\wh{Q}$ is harmonic on $\C\setminus S$, (iii) one has that
\begin{equation}\label{app0}\wh{Q}(\zeta)=\log\babs{\zeta}^2+O(1)\quad \text{as}\quad z\to\infty.\end{equation}
For proofs of these statements, we refer to \cite{ST}, Theorem I.4.7 and \cite{HM2}.

Fix a point $z\in S$ with $\dist(z,\d S)\ge 3\delta_n$.
We can here take
$\delta_n=M\sqrt{\log n/n}$ for some large $M$, but
any fixed positive sequence with $n\delta_n^3\to 0$ and $\liminf_{n\to\infty}n\delta_n^2/\log n$ large enough
will also work.
Recall that
$$\psi(z,\zeta)=Q(z)+\d Q(z)\cdot(\zeta-z)+\frac 1 2
\d^2Q(z)\cdot (\zeta-z)^2.$$
Put
$$k_z^\#(\zeta)=n\Delta Q(z)\e^{n\psi(z,\bar{\zeta})}.$$
Observe that, by Taylor's formula,
\begin{equation}\label{app1}
\overline{k_z^\#(\zeta)}\e^{-nQ(\zeta)}=n\Delta Q(z)
\e^{-nH_z(\zeta)-n\Delta Q(z)\babs{z-\zeta}^2}\e^{n\eps_z(\zeta)},\end{equation}
where
$$H_z(\zeta)=\d Q(z)(\zeta-z)+\frac 1 2 \d^2 Q(z)(\zeta-z)^2$$
and $\eps_z(\zeta)=O(\babs{z-\zeta}^3)$.

Let $\chi_z=\chi_{z,n}$ be a sequence of
cut-off functions with $\chi_z=1$ in $D(z;3\delta_n/2)$
and $\chi_z=0$ outside $D(z;2\delta_n)$, and also
$\|\dbar \chi_z\|_{L^2}\le C$. In the following, we write
$\|f\|_{nQ}^2=\int_\C\babs{f}^2\e^{-nQ}\dA$.

Let $f$ be holomorphic in $D(z;2\delta_n)$. Observe that (since $\e^x=1+O(x)$ as $x\to 0$)
$$\langle \chi_z\cdot f,k_z^\#\rangle_{nQ}=\lpar 1+O(n\delta_n^3)\rpar n\Delta Q(z)\int \chi_z\cdot F_n
\e^{-n\Delta Q(z)\babs{\zeta-z}^2}\dA(\zeta),$$
where we have put
$$F_n=f\cdot \e^{-nH_z}.$$

It yields
$$\langle \chi_z\cdot f,k_z^\#\rangle_{nQ}=
-(1+O(n\delta_n^3))\int
\frac {\chi_z(\zeta)\cdot F_n(\zeta)}
{\zeta-z}\cdot
\dbar_\zeta\lpar\e^{-n\Delta Q(z)\babs{\zeta-z}^2}\rpar
\dA(\zeta),$$
or by Cauchy's formula,
$$\langle \chi_z\cdot f,k_z^\#\rangle_{nQ}=
(1+O(n\delta_n^3))\lpar F_n(z)+
\int_{\babs{\zeta-z}\ge \delta_n}
\frac {\dbar\chi_z(\zeta)\cdot F_n(\zeta)}
{\zeta-z}\e^{-n\Delta Q(z)\babs{\zeta-z}^2}\dA(\zeta)
\rpar.$$
Since $F_n(z)=f(z)$, it yields
(using Cauchy--Schwarz)
$$\babs{f(z)-\langle \chi_z\cdot f,k_z^\#\rangle_{nQ}}
\le Cn\delta_n^3\babs{f(z)}
+\delta_n^{-1}\e^{-n\delta_n^2}
\left\|\dbar\chi_z\right\|_{L^2}\cdot
\lpar \int_{D(z;2\delta_n)}
\babs{F_n(\zeta)}^2\e^{-n\Delta Q(z)
\babs{z-\zeta}^2}\dA(\zeta)\rpar^{1/2}.
$$

But in view of Lemma \ref{subh}, we have an estimate
$$\babs{f(z)}\precsim
\sqrt{n}\left\|f\right\|_{nQ}\e^{nQ(z)/2}.$$
Observing that $\babs{F_n(\zeta)}^2\e^{-n\Delta Q(z)\babs{z-\zeta}^2}\precsim
\babs{f(\zeta)}^2\e^{-nQ(\zeta)}\e^{nQ(z)}$
when $\zeta\in D(z;2\delta_n)$, we conclude
\begin{equation}\label{app2}
\babs{f(z)-\langle \chi_z\cdot f,k_z^\#\rangle_{nQ}}\le
Cn^{3/2}\delta_n^3\|f\|_{nQ}\e^{nQ(z)/2}.\end{equation}

 Let $k_w(z)=K_n(z,w)$ be the reproducing kernel for the subspace $H_n$ of $L^2_{nQ}=L^2(\e^{-nQ},\dA)$ consisting of all analytic polynomials of degree $\le n-1$; let $P_n f(z)=\langle f,k_z\rangle_{nQ}$ and
$P_n^\# f(z)=\langle f,k_z^\#\rangle_{nQ}$.
Then
$$P_n^\#\left[\chi_z\cdot k_w\right](z)=\langle
\chi_z\cdot k_w,k_z^\#\rangle_{nQ}=\overline{\langle
\chi_zk_z^\#,k_w\rangle_{nQ}}=\overline{
P_n\left[\chi_z\cdot k_z^\#\right](w)}.$$
This gives (since $k_z(w)=\overline{k_w(z)}$)
$$\babs{k_z(w)-P_n\left[\chi_z\cdot k_z^\#\right](w)}
=\babs{k_w(z)-P_n^\#\left[\chi_z\cdot k_w\right](z)},$$
and, since $\|k_w\|_{nQ}\precsim \sqrt{n}\e^{nQ(w)/2}$,
we get from \eqref{app2} that
\begin{equation}\label{app3}\babs{k_z(w)-P_n\left[\chi_z\cdot k_z^\#\right](w)}
\le Cn^2\delta_n^3\e^{nQ(z)/2}\e^{nQ(w)/2}.\end{equation}

To finish the proof of Lemma \ref{l1}, fix $z\in S_n$ and $w\in D(z;\delta_n)$ and let
$$u(\zeta)=\chi_z(\zeta)k_w^\#(\zeta)-P_n\left[\chi_z
k_w^\#\right](\zeta).$$
By a well known version of H\"{o}rmander's estimate for the
$L^2$-minimal solution to the $\dbar$-equation (see \cite{B} or \cite{A}, \textsection 5.2) we have a pointwise estimate
$$\babs{u(z)}\le Cn\e^{-cn\delta_n^2}\e^{n(Q(z)+Q(w))/2},$$
where $C$ and $c$ are positive constants.

Since $n\e^{-cn\delta_n^2}=o(1)$, we obtain
$$\babs{K_n(z,w)-k_w^\#(z)}\e^{-n(Q(z)+Q(w))/2}\le
Cn^2\delta_n^3.$$
To finish the proof of Lemma \ref{l1}, it now suffices to
recall that
$$\mathbf{K}_n(z,w)=K_n(z,w)\e^{-n(Q(z)+Q(w))/2}
\quad \text{and}\quad
\mathbf{K}_n^\#(z,w)=k_z^\#(w)\e^{-n(Q(z)+Q(w))/2}.$$

\end{document}